\documentclass[12pt,reqno]{amsart}
  \usepackage[all,2cell,arrow]{xy}
  \usepackage{amsfonts}
  \usepackage{amsthm}
  \usepackage{amsmath}
  \usepackage{amssymb}
\usepackage{mathtools}
\usepackage{tensor}
 \numberwithin{equation}{section}
 \usepackage{xcolor}
\usepackage{graphicx}
\usepackage[hidelinks]{hyperref}
\usepackage{enumerate,xspace}

\UseAllTwocells
\CompileMatrices

\renewcommand{\epsilon}{\varepsilon}
\renewcommand{\phi}{\varphi}

\newcommand{\ca}{\ensuremath{\mathcal A}\xspace}
\newcommand{\cb}{\ensuremath{\mathcal B}\xspace}
\newcommand{\cc}{\ensuremath{\mathcal C}\xspace}
\newcommand{\cd}{\ensuremath{\mathcal D}\xspace}

\newcommand{\cf}{\ensuremath{\mathcal F}\xspace}
\newcommand{\cg}{\ensuremath{\mathcal G}\xspace}

\newcommand{\ci}{\ensuremath{\mathcal I}\xspace}
\newcommand{\cj}{\ensuremath{\mathcal J}\xspace}
\newcommand{\ck}{\ensuremath{\mathcal K}\xspace}
\newcommand{\cl}{\ensuremath{\mathcal L}\xspace}
\newcommand{\cm}{\ensuremath{\mathcal M}\xspace}

\newcommand{\cv}{\ensuremath{\mathcal V}\xspace}
\newcommand{\cw}{\ensuremath{\mathcal W}\xspace}

\newcommand{\bbi}{\ensuremath{\mathbb I}\xspace}
\newcommand{\bbn}{\ensuremath{\mathbb N}\xspace}

\newcommand{\atwo}{{\mathbf 2}}

\newcommand{\Mod}{\ensuremath{\mathbf{Mod}}\xspace}

\newcommand{\Cat}{\ensuremath{\mathbf{Cat}}\xspace}

\newcommand{\Set}{\ensuremath{\mathbf{Set}}\xspace}
\newcommand{\SSet}{\ensuremath{\mathbf{SSet}}\xspace}
\newcommand{\qCat}{\ensuremath{\mathbf{qCat}}\xspace}

\newcommand{\Cart}{\ensuremath{\mathbf{Cart}}\xspace}
\newcommand{\coCart}{\ensuremath{\mathbf{coCart}}\xspace}
\newcommand{\DiscCart}{\ensuremath{\mathbf{DiscCart}}\xspace}
\newcommand{\DisccoCart}{\ensuremath{\mathbf{DisccoCart}}\xspace}

\newcommand{\iso}{\bbi}

\newcommand{\co}{\ensuremath{^{\textnormal{co}}}}
\newcommand{\op}{\ensuremath{^{\textnormal{op}}}}

\newcommand{\coop}{\ensuremath{^{\textnormal{co,op}}}}

\newcommand{\fib}{\ensuremath{_{\textnormal{fib}}}}

\DeclareMathOperator{\cod}{cod}
\DeclareMathOperator{\dom}{dom}

\DeclareMathOperator{\lali}{\mathbf{Rari}}

\DeclareMathOperator{\lari}{\mathbf{Lari}}

\DeclareMathOperator{\TF}{\mathbf{TF}}

\DeclareMathOperator{\WE}{Equiv}
\DeclareMathOperator{\colim}{colim}
\DeclareMathOperator{\Fib}{\mathbf{Fib}}

\newcommand{\two}{\ensuremath{\mathbf{2}\xspace}}

\newcommand{\x}{\times}

\def\1c#1{\stackrel{#1}{\to}}

  \newtheorem{proposition}{Proposition}[section]
  \newtheorem{lemma}[proposition]{Lemma}
  
  \newtheorem{corollary}[proposition]{Corollary}
    
  \newtheorem{theorem}[proposition]{Theorem}

  \theoremstyle{definition}
  \newtheorem{definition}[proposition]{Definition}

  \newtheorem{example}[proposition]{Example}
  \newtheorem{examples}[proposition]{Examples}

  \theoremstyle{remark}
  \newtheorem{remark}[proposition]{Remark}

  \newcounter{c}
  \renewcommand{\[}{\setcounter{c}{1}$$}
  \newcommand{\etyk}[1]{\vspace{-7.4mm}$$\begin{equation}\Label{#1}
  \addtocounter{c}{1}}
  \renewcommand{\]}{\ifnum \value{c}=1 $$\else \end{equation}\fi}
\newcommand\isf{{\mathrel{\rotatebox{270}{\scriptsize$\twoheadrightarrow$}}}}

\begin{document}

\title[Accessible $\infty$-cosmoi]{Accessible $\infty$-cosmoi}
\author{John Bourke}
\address{Department of Mathematics and Statistics, Masaryk University, Kotl\'a\v rsk\'a 2, Brno 61137, Czech Republic}
\email{bourkej@math.muni.cz}

\author{Stephen Lack}
\address{School of Mathematical and Physical Sciences, Macquarie
 University NSW 2109, Australia}
\email{steve.lack@mq.edu.au}
\thanks{The first-named author acknowledges the support of the Grant Agency of the Czech Republic under the grants 19-00902S and 22-02964S. The second-named author acknowledges with gratitude the support of an
 Australian Research Council Discovery Project DP190102432.}

\date{August 31, 2022}

\begin{abstract}
We introduce the notion of an accessible $\infty$-cosmos and prove
that these include the basic examples of $\infty$-cosmoi and are
stable under the main constructions.  A consequence is that the vast
majority of known examples of $\infty$-cosmoi are accessible. By the
adjoint functor theorem for homotopically enriched categories which
we proved in an earlier paper, joint with Luk\'{a}\v{s}
Vok\v{r}\'{i}nek, it follows, for instance, that all such 
$\infty$-cosmoi have flexibly weighted homotopy colimits. 
\end{abstract}

\maketitle 

\tableofcontents

\section{Introduction}

The theory of $\infty$-categories has experienced an explosion of interest in
recent years, due in part to the needs of researchers in various areas
of geometry, topology, logic, and mathematical physics. Multiple
approaches have appeared, leading to multiple definitions of
$(\infty,1)$-category, or multiple {\em models} in the usual parlance,
since each of these is seen as being only some sort of presentation of
the ``true'' notion. Prominent examples of these models include quasicategories,
complete Segal spaces, and Segal categories. Each of these, as
well as various others, has its own distinct flavour, coming with
various resulting advantages and disadvantages. 

Substantial progress has been made in the comparison between these
models --- see \cite{BergnerBook} for a recent survey --- but it is
also natural to hope for a ``model independent'' approach. Over a
number of years Riehl and Verity have been developing one such
approach, under the name of {\em $\infty$-cosmos}, and their theory
has now reached a high level of power and sophistication. For an
introduction to many aspects of this theory, see their book \cite{elements}.

The theory of $\infty$-cosmoi is very much homotopical in nature. 
In an earlier paper \cite{BourkeLackVokrinek} joint with Luk\'{a}\v{s}
Vok\v{r}\'{i}nek, we proved a very general homotopical adjoint
functor theorem for enriched categories. This actually included
Freyd's General Adjoint Functor Theorem (GAFT) as the special case of
$\Set$-enriched (i.e. unenriched) homotopically trivial categories, but the
main motivation was the study of $\infty$-cosmoi, which are in fact
certain simplicially enriched categories.

In the case of ordinary (unenriched, homotopically trivial) categories, the
solution set condition which appears in the GAFT can be
simplified using the theory of accessible categories
\cite{MakkaiPare,AR}. We found similar simplifications were available
in the enriched homotopical setting of \cite{BourkeLackVokrinek}, and
included versions of our main results which were formulated using
(enriched) accessible categories. 

Among other things, we proved the following:

\begin{theorem}
  Let $U\colon\cl\to\ck$ be an accessible cosmological functor between
  $\infty$-cosmoi which are accessible simplicially enriched
  categories. Then $U$ has a homotopical left adjoint. 
\end{theorem}

\begin{theorem}
  Let $\ck$ be an $\infty$-cosmos which is accessible as a
  simplicially enriched category. Then $\ck$ has flexibly weighted homotopy
  colimits. 
\end{theorem}

The second of these is in fact a straightforward consequence of the
first (much as in the usual unenriched, homotopically trivial case). As such
it is an analogue of the classical result that an accessible category
which is complete is also cocomplete; indeed these are both special
cases of a single theorem \cite[Theorem~8.9]{BourkeLackVokrinek}.

It is significant since the definition of $\infty$-cosmos involves the
existence of various sorts of limit, but not of any colimits. Colimits
can be used for various things such as the formation of Kleisli
objects and localizations, while left adjoints as in the
first theorem could be used for example to construct free completions of
$\infty$-categories under some class of limits or colimits. 

Of course the interest in the two theorems quoted would be slight if there were
not a good supply of $\infty$-cosmoi which were accessible in the
relevant sense. We described some examples in
\cite{BourkeLackVokrinek}, all firmly based in the world of
quasicategories, but promised to expand this list in a future
paper. This is what we shall do here. 
In fact we introduce a notion of accessibility for $\infty$-cosmoi
which is stronger than that of the earlier paper; we do this because
of the good stability properties it enjoys, which allow us to
construct many new examples of accessible $\infty$-cosmoi from any
given one. 

An $\infty$-cosmos is a universe in which one can develop the theory
of $\infty$-categorical structures, much as a (suitably endowed)
2-category is a universe in which one can develop the theory of
categorical structures.  
Given an $\infty$-cosmos $\ck$, whose objects are referred to as $\infty$-categories,
 there are further $\infty$-cosmoi of
\begin{itemize}
\item $\infty$-categories with limits or colimits of some type (or
  a combination of both)
\item isofibrations of $\infty$-categories (analogous to the usual
  ``arrow categories'')
\item various flavours of fibration of $\infty$-categories
  (cartesian or cocartesian, 1-sided or 2-sided, discrete or not).
\end{itemize}

These constructions of new $\infty$-cosmoi from old are all described
in \cite{elements}; what we do here is show that if the original
$\infty$-cosmos is accessible, in our sense, then so is each of the
resulting ones.

In particular, we could take as our starting $\infty$-cosmos that
consisting of the quasicategories, the compete Segal spaces, or the
Segal categories: each of these is accessible.

A precursor to the present work is \cite{Bourke2019Accessible}, which deals
with $2$-categories of categorical structures rather than $\infty$-cosmoi of $\infty$-categorical
structures.  In that setting, accessibility is seen to be closely
related to weakness --- for instance,
 the 2-category of monoidal categories and strong monoidal
functors is accessible, but the full sub-2-category consisting of
strict monoidal categories is not.  %it holds for monoidal but not strict monoidal structure.
One of the guiding ideas in the present work is that
since in the $\infty$-categorical world we are primarily interested in weak structures, the
vast majority, if not all, the examples of interest should form accessible $\infty$-cosmoi.

We now turn to an outline of the paper. We begin in
Section~\ref{sect:prelim} with a brief review of the necessary background on $\infty$-cosmoi and accessible
categories, both ordinary and simplicial. Then in
Section~\ref{sect:defn} we introduce our main concept of accessible
$\infty$-cosmos, and show that these include the basic examples
arising from suitable simplicially enriched model categories.
In Section~\ref{sect:first}, we study a first raft of closure properties
of accessible $\infty$-cosmoi, including $\infty$-cosmoi of isofibrations,
slices and duals of $\infty$-cosmoi, and pullbacks of cosmological embeddings.
The technical heart of the paper is Section~\ref{sect:lali},
where we show that for an accessible $\infty$-cosmos $\ck$, the
$\infty$-cosmos $\lali(\ck)$ of {\em left adjoint left inverses} in
$\ck$ is also accessible. In Section~\ref{sect:TF} we prove the
corresponding fact about {\em trivial fibrations} in $\ck$, with further results on equivalences.
In Section~\ref{sect:Apps} we use the results of the previous three
sections to deduce all our remaining closure properties for accessible
$\infty$-cosmoi.

\section{Preliminaries}\label{sect:prelim}

In this section, we run through the key concepts needed later in the paper.

\subsection{The Joyal model structure}\label{sect:joyal}

Of importance in the theory of $\infty$-cosmoi is the \emph{Joyal model structure} on the category $\SSet$ of simplicial sets.  This is a combinatorial model structure whose cofibrations are the monomorphisms and whose fibrant objects are the quasicategories, and makes $\SSet$ into a cartesian closed model category.  

The model structure has generating cofibrations the boundary
inclusions $\partial \Delta^n \hookrightarrow \Delta^n$.  The
fibrations between quasicategories will be called \emph{isofibrations
  of quasicategories} and can be characterised as those morphisms
having the right lifting property with respect to the inner horn
inclusions $ \Lambda^n_k \to \Delta^n$ together with the endpoint
inclusions $1 \to \iso$, where $\iso$ is the nerve of the free-living
isomorphism.  The weak equivalences between quasicategories will often
be referred to as \emph{equivalences of quasicategories}.

\subsection{$\infty$-cosmoi}\label{sect:infty-cosmos}

An {\em $\infty$-cosmos} is a simplicially enriched category $\ck$
together with a class of morphisms called \emph{isofibrations}, denoted $A \twoheadrightarrow B$, closed
under composition and containing the isomorphisms, such that 
\begin{enumerate}
\item each hom $\ck(A,B)$ is a quasicategory;
\item the morphism $\ck(A,p)\colon\ck(A,B)\twoheadrightarrow\ck(A,C)$
  is an isofibration of quasicategories for each isofibration $p
  \colon B \twoheadrightarrow C$;
\item \ck has products, powers by simplicial sets, pullbacks along
  isofibrations, and limits of countable towers of isofibrations;
\item the class of isofibrations is closed under these limits, under
  Leibniz powers by monomorphisms of simplicial sets, and contains all
  maps with terminal codomain.  
\end{enumerate}
Following \cite{elements}, we call the various
limits appearing in (3) above {\em cosmological limits}.  

\begin{examples}
 (a) A simple example of an $\infty$-cosmos is the
$2$-category $\Cat$ of (small) categories.  Like any $2$-category,
this can be viewed as a simplicially enriched category by taking the
nerves of its hom-categories.  Isofibrations between categories are
those functors $F\colon\cc \to \cd$ having the isomorphism lifting
property: namely, given an isomorphism $f\colon A \to FB \in \cd$,
there exists an isomorphism $f'\colon A' \to B \in \cc$ such that $Ff'=f$.  As explained in Proposition~1.2.11 of \cite{elements}, with this choice of isofibrations $\Cat$ becomes a $\infty$-cosmos.

(b) A fundamental example is the $\infty$-cosmos $\qCat$
of quasicategories, which is the full simplicially enriched
subcategory of $\SSet$ with objects the quasicategories, and with
isofibrations as described in Section~\ref{sect:joyal} --- see
Proposition~1.2.10 of \cite{elements}.
\end{examples}

Each $\infty$-cosmos $\ck$ has a \emph{homotopy $2$-category} $h\ck$.
This has the same objects as $\ck$ and hom-categories $h\ck(A,B)=\pi(\ck(A,B))$ where $\pi$ is the left adjoint to the nerve functor.  An important example is the \emph{$2$-category of quasicategories} \cite{Riehl2015The-2-category}, which is the homotopy $2$-category $h\qCat$.

A morphism $f\colon B\to C$ in $\ck$ is said to be an {\em
  equivalence} if the induced $\ck(A,f)\colon\ck(A,B)\to\ck(A,C)$ is an
equivalence of quasicategories for all $A\in\ck$, and a trivial fibration if it is both an equivalence and an isofibration.

We write $\ck_0$ for the underlying ordinary category of a
simplicially enriched category $\ck$. We write $\ck^\two$ for
the simplicially enriched category of arrows in $\ck$. We define full
subcategories of $\ck^\two$ as follows:
\begin{itemize}
\item $\ck^\isf$ consists of the isofibrations
\item $\WE(\ck)$ consists of the equivalences
\item $\TF(\ck)=\WE(\ck)\cap\ck^\isf$ consists of the trivial fibrations.
\end{itemize}

A \emph{cosmological functor} $F\colon\ck \to \cl$ between $\infty$-cosmoi is a simplicially enriched functor which preserves isofibrations and cosmological limits.

\subsection{Weighted limits, colimits, and their homotopy versions}

In addition to cosmological limits, we now run through the various kinds of (weighted homotopy) limits and colimits that we encounter in the present paper.  

 Let $\cc$ be a small simplicially enriched category and consider the
 enriched functor category $[\cc,\SSet]$, whose objects are called
 weights.  In larger diagrams, we will sometimes denote hom-objects
 $[\cc,\SSet](F,G)$ by $(F,G)$.   Let $W \in [\cc,\SSet]$ be a weight.
 Given a diagram $S\colon\cc \to \ck$ in a simplicially enriched
 category $\ck$ a weighted limit $L$ is defined by a cone $\eta\colon W \to \ck(L,S-)$ for which the induced morphism
\begin{equation}\label{eq:cone}
\xymatrix{
\ck(A,L) \ar[r] & [\cc,\SSet](W,\ck(A,S-)) } 
\end{equation}
is invertible. A weighted colimit is a weighted limit in $\ck\op$.

\subsubsection{Flexible limits}
 Flexible limits and cofibrantly-weighted limits are those whose defining weights are flexible or cofibrant.  To understand them, we observe that the Joyal model structure on $\SSet$ induces the enriched projective model structure on $[\cc,\SSet]$  by Proposition~A.3.3.2 and Remark~A.3.3.4 of \cite{Lurie}, and this has generating cofibrations
  $$\ci = \{\partial \Delta^n \times \cc(X,-) \to \Delta^n \times  \cc(X,-)\colon n \in \mathbb N, X \in \cc\}$$ Riehl and Verity's flexible weights\footnote{
    This differs from the usage in 2-category theory, where ``flexible'' is taken to mean cofibrantly-weighted.} are precisely the $\ci$-cellular weights, and so --- in particular --- cofibrant weights.  Flexibly weighted limits in an $\infty$-cosmos, or just flexible limits, are of importance since, by Proposition 6.2.8(i) of \cite{elements}, each $\infty$-cosmos $\ck$ admits them.

\subsubsection{Weighted colimits that are homotopy colimits}
Let us now turn to the question of when colimits are homotopy colimits
--- here we emphasise the colimit point of view which will be our
primary interest.  Consider a weight $W\colon\cc\op\to\SSet$, not necessarily cofibrant,
but suppose now that $\ck$ is locally fibrant and
$S\colon\cc\to\ck$, and let $p\colon Q \to W$ be a cofibrant replacement of $W$.  The weighted colimit $W*S$ is said to be a \emph{homotopy colimit} if the induced morphism
\begin{equation*}
\xymatrix{
\ck(W*S,A) \ar[r]^-{\cong} &  (W,\ck(S-,A)) \ar[r]^{p^*} & (Q,\ck(S-,A))
}
\end{equation*} 
is an equivalence of quasicategories. 
Note that this is equally to say that the second component 
\[ p^*\colon [\cc\op,\SSet](W,\ck(S-,A)) \to [\cc\op,\SSet](Q,\ck(S-,A)) \]
is an equivalence of quasicategories.  
Since $\ck$ is locally fibrant it follows that the property of being a
homotopy colimit is independent of the choice of cofibrant replacement
so that, in testing for homotopy colimits, we are free to assume that
$Q$ is flexible and that $p\colon Q \to W$ is a trivial fibration.

Of particular interest in this paper is the case  
where $\cc$ is a small $\lambda$-filtered category and ${
  W}=\Delta 1$ is the terminal weight, so that { $W*S$} is the $\lambda$-filtered colimit of $S$.  Specialising the above situation, we can thus speak of $\lambda$-filtered colimits being homotopy colimits.  

Dually, a limit $\{W,S\}$ is a homotopy limit if it is a homotopy colimit in $\ck\op$.

\subsection{Accessible categories} 

We now turn to some basic results about accessible
categories. For further information, see \cite{AR} or \cite{MakkaiPare}.

A category $\ck$ is {\em
  $\lambda$-accessible}, for a regular cardinal $\lambda$, just when it
is the free completion under $\lambda$-filtered colimits of a small
category. More concretely, this will be the case when $\ck$ has
$\lambda$-filtered colimits, and there is a small full subcategory
$\cg$ whose objects are $\lambda$-presentable and such that every
object of $\ck$ is a $\lambda$-filtered colimit of objects in $\cg$.

A category is {\em accessible} if it is $\lambda$-accessible for some
regular cardinal $\lambda$.

An accessible category is complete if and only if it is cocomplete, in
which case it is said to be a {\em locally presentable} category.

A functor between accessible categories is said to be accessible if it
preserves $\lambda$-filtered colimits for some regular cardinal $\lambda$.

\begin{example}\label{ex:adj-acc}  (See \cite[Proposition~2.3]{AR}.)
  Any left or right adjoint between accessible categories is an
  accessible functor. 
\end{example}

The Makkai-Par\'e Limit Theorem \cite[Theorem~5.1.6]{MakkaiPare}  asserts that the 2-category of
accessible categories, accessible functors, and natural
transformations has bicategorical limits (bilimits), and these are formed at the
level of underlying ordinary categories.  For our purposes, important examples of bilimits include products, powers by small categories (functor categories in $\Cat$) as well as comma objects, of which special cases are slice categories. 
We shall also apply the Makkai-Par\'e Limit Theorem in the
following special case.

\begin{proposition}\label{prop:pb-acc}
  Consider a pullback of categories
  \[ \xymatrix{
      \ca \ar[r]^-P \ar[d]_Q & \cb \ar@{>>}[d]^F \\ \cc \ar[r]_-G & \cd
    } \]
  in which $F$ and $G$ are accessible functors between accessible
  categories, and $F$ is an isofibration of categories. Then $\ca$ is an accessible
  category and $P$ and $Q$ are accessible functors. 
\end{proposition}

\proof
The assumption that $F$ is an isofibration means that this pullback is
also a bipullback \cite{JS-pseudopullbacks}, and now the result
follows by the Limit Theorem.
\endproof

In this paper, we are primarily interested in simplicially enriched
categories.  A simplicially enriched category is $\lambda$-accessible just
when it is the free completion of a small (simplicially enriched)
category under (enriched) $\lambda$-filtered colimits%
\footnote{ For general $\cv$ another notion of enriched accessibility is considered in \cite{BorceuxQuinteiro}.  However, in the special case of simplicial enrichment, it is equivalent to the notion described above by \cite[Theorem~3.14]{LackTendas2021}. }.  In the simplicially enriched categories
of interest to us each object $A$ moreover has a power (also known as cotensor) 
$X\pitchfork A$ by each simplicial set $X$.  In this context there
are simpler descriptions of accessibility, as described in the following proposition,
which follows immediately from \cite[Proposition~8.11]{BourkeLackVokrinek}.

\begin{proposition}\label{prop:enr-acc}
  For a simplicially enriched category $\ck$ with powers, the
  following are equivalent:
  \begin{enumerate}
  \item $\ck$ is $\lambda$-accessible as a simplicially enriched category;
  \item the underlying ordinary category $\ck_0$ is
    $\lambda$-accessible, and the hom-functor
    $\ck(A,-)\colon\ck_0\to\SSet$ is $\lambda$-accessible for each
    $\lambda$-presentable object of $\ck_0$;
    % \item $\ck_0$ is an accessible category and
    %   $X\pitchfork-\colon\ck_0\to\ck_0$ is an accessible functor for
    %   each $X\in\SSet$;
  \item $\ck_0$ is an accessible category and
    $\Delta[n]\pitchfork-\colon\ck_0\to\ck_0$ is an accessible functor for
    each $n\in\bbn$. 
  \end{enumerate}
  For such a $\ck$, the functors $\ck(A,-)\colon\ck_0\to\cv_0$
  and $X\pitchfork-\colon\ck_0\to\ck_0$ are accessible for all
  $A\in\ck$ and $X\in\cv$.
\end{proposition}

\section{Accessible $\infty$-cosmoi}\label{sect:defn}

\begin{definition}\label{defn:main}

An $\infty$-cosmos $\ck$ is said to be \emph{accessible} if 
\begin{enumerate}
\item $\ck$ is accessible as a simplicially enriched category;
% \item its underlying category $\ck_0$ is accessible;
% \item for each simplicial set $X$, the functor
%   $X\pitchfork-\colon\ck_0\to\ck_0$ sending 
%   an object $A\in\ck$ to the power $X\pitchfork A$ is accessible;
\item $\ck^\isf_0$ is accessible and accessibly embedded in
  $\ck^\two_0$; \label{item:isofibs-accessible}
\item There exists a regular cardinal $\lambda$ such that $\lambda$-filtered colimits exist in $\ck$ and are homotopy colimits.\label{item:hty-colims}
\end{enumerate}
 A cosmological functor between accessible $\infty$-cosmoi is
accessible if its underlying functor is accessible.
\end{definition}

\begin{remark}
  In \cite{BourkeLackVokrinek} an $\infty$-cosmos was said to be
  accessible when it is so as a simplicially enriched category, our
  Condition (1). % by
  % Proposition~\ref{prop:enr-acc}, this
  % is equivalent to the first two conditions only.
  As mentioned in the introduction, and as anticipated in a
  footnote in \cite[Section~9.4]{BourkeLackVokrinek}, we have strengthened the definition here so that the
  class of accessible $\infty$-cosmoi has better stability properties,
  such as Proposition~\ref{prop:fib-acc} below.  In fact
  Proposition~\ref{prop:fib-acc} and the other results in
  Section~\ref{sect:first} would all hold if we added only Condition~\eqref{item:isofibs-accessible}; it is in order to prove the accessibility of $\lali(\ck)$ for
  an accessible $\infty$-cosmos $\ck$ that we include  Condition~\eqref{item:hty-colims} in
  the definition. For further comments on Condition~~\eqref{item:hty-colims},
  see Remark~\ref{rmk:open-problem}. 
\end{remark}

\begin{remark}
In this paper we are largely avoiding the question of indices of
accessibility, but perhaps a few words are appropriate. If $\ck$ is $\lambda$-accessible as a simplicially-enriched category,
 it follows by the characterization in Proposition~\ref{prop:enr-acc}
 that it is also $\lambda'$-accessible for any
 $\lambda'\triangleright\lambda$, in the sense of
 \cite[Definition~2.12]{AR}. Similarly, if $\ck^\isf_0$ is
 $\lambda$-accessible closed in $\ck^\two_0$ under $\lambda$-filtered
 colimits, then the same is true for any
 $\lambda'\triangleright\lambda$.
 On the other hand, if $\ck$ satisfies
 Condition~\eqref{item:hty-colims} for a given $\lambda$, then it does
 so for any $\lambda'>\lambda$. By the Uniformization
 Theorem~\cite[Theorem~2.19]{AR}, if $\ck$ is an accesssible
 $\infty$-cosmos, there are arbitrarily large $\lambda$ for which
 $\ck$ is $\lambda$-accessible as an enriched category, $\ck^\isf_0$ is
 accessible and accessibly embedded in $\ck^\two_0$, and
 Condition~\eqref{item:hty-colims} holds for the given $\lambda$.
\end{remark}

In addition to the accessibility of powering functors, further exactness properties are easily seen to hold in an accessible $\infty$-cosmos.

\begin{lemma}\label{lem:limits}
Let $\ck$ be an $\infty$-cosmos for which $\ck$ is $\lambda$-accessible as a simplicially enriched category.  Then pullbacks of isofibrations and finite products commute with $\lambda$-filtered colimits in $\ck$.% \blue
% \grey In an accessible $\infty$-cosmos $\ck$, pullbacks of isofibrations and finite products commute with $\lambda$-filtered colimits for some regular cardinal $\lambda$.\black
\end{lemma}
\begin{proof}
% \grey Suppose that $\ck$ is $\lambda$-accessible as a $\SSet$-enriched
% category and \black
Let $J\colon\ck_\lambda\to\ck$ be the inclusion of the
full subcategory consisting of the $\lambda$-presentable objects.
The induced functor $\ck(J-,1)\colon\ck\to[\ck\op_\lambda,\SSet]$
is fully faithful, preserves $\lambda$-filtered colimits, and
preserves any existing limits.  Since it reflects isomorphisms, $\ck(J-,1)$ also reflects any commutativities between existing limits and $\lambda$-filtered colimits that hold in $[\ck\op_\lambda,\SSet]$.  Since $\SSet$ is locally finitely presentable, both pullbacks 
and finite products commute with filtered colimits in $\SSet$ and so in $[\ck\op_\lambda,\SSet]$, as required.
\end{proof}

In the following sections, we will show that various
constructions applied to accessible $\infty$-cosmoi yield new
accessible $\infty$-cosmoi. But for this to have any interest, one needs a stock
of initial examples. The next result provides some.

\begin{proposition}
  Let $\cm$ be a simplicially enriched category, equipped with a
  combinatorial model structure which is enriched with respect to the
  Joyal model structure on $\SSet$. Suppose that every fibrant object
  of $\cm$ is also cofibrant. Then the full subcategory
  $\cm\fib$ of $\cm$ consisting of the fibrant objects is an
  accessible $\infty$-cosmos, in which the isofibrations and
  equivalences are the fibrations and weak equivalences (between fibrant objects in each case) of the model structure.   Moreover, the inclusion $\cm\fib \hookrightarrow \cm$ is an accessible embedding.
\end{proposition}

\proof
$\cm\fib$ is an $\infty$-cosmos by \cite[Proposition~E.1.1]{elements} and accessible as a simplicially enriched category by
\cite[Proposition~9.1]{BourkeLackVokrinek}.  The same result shows that $\cm\fib \hookrightarrow \cm$ is an accessible embedding.

Now the full subcategories $\cf,\cw \hookrightarrow \cm^{\atwo}$ of fibrations and weak equivalences are accessible and accessibly embedded.  Therefore, by Proposition~\ref{prop:pb-acc}, the pullback $\cf \cap \cm\fib \hookrightarrow \cm\fib^{\atwo}$ is accessible and accessibly embedded, establishing Condition~\eqref{item:isofibs-accessible}.

Now since $\cf$ and $\cw$ are accessibly embedded, there exists 
a  regular cardinal  $\lambda$ such that both $\cf$ and $\cw$ are closed in
$\cm^{\atwo}$  under $\lambda$-filtered colimits;  this 
closure property of $\cf$   also ensures that $\cm\fib \hookrightarrow \cm$ is
closed under $\lambda$-filtered colimits.  We will use these
assumptions to prove that $\lambda$-filtered colimits in
$\cm\fib$ are homotopy colimits, thereby establishing Condition~\eqref{item:hty-colims}.

To this end, let $\cc$ be a small $\lambda$-filtered category.  Since $\cm$ is cocomplete as an enriched category, we can consider the weighted colimit functor
\begin{equation*}
\xymatrix{
-*-:[\cc\op,\SSet]_0 \times [\cc,\cm]_0 \to \cm_0
}
\end{equation*}

Both $\SSet$, equipped with the Joyal model structure, and $\cm$ are
combinatorial model categories.  Therefore $[\cc\op,\SSet]_0$ and
$[\cc,\cm]_0$ each admit both the projective and injective model
structure.  The key result for us is that when one of these is
equipped with the projective model structure and the other with the
injective model structure, the weighted colimit functor becomes a left
Quillen bifunctor.  (This is a special case of Theorem~C.3.13
 of \cite{elements}, which generalises Gambino's result \cite{Gambino} for the Kan-Quillen model structure.)

Now consider a diagram $S\colon\cc \to \cm$  taking values among the
fibrant objects.  Since $\cm\fib$ is closed under
$\lambda$-filtered colimits, the colimit $\Delta 1 * S$ is also
fibrant.  Let $p\colon Q \Delta 1 \to \Delta 1$ in
$[\cc\op,\SSet]$ and $q\colon QS \to S$ in $[\cc,\cm]$ be projective cofibrant replacements.  We then have a commutative square
\begin{equation*}
\xymatrix{
Q\Delta 1 * QS \ar[d]_{p * QS} \ar[rr]^{Q\Delta 1 * q} && Q\Delta 1 * S \ar[d]^{p * S} \\
\Delta 1 * QS \ar[rr]_{\Delta 1 * q} && \Delta 1 * S
}
\end{equation*}
in $\cm$.  

Now since $S$ is pointwise fibrant, it is pointwise cofibrant and
therefore injectively cofibrant.  Hence $q\colon QS \to S$ is a weak
equivalence between injectively cofibrant objects.  Since $Q\Delta 1$
is projectively cofibrant and $-*-$ a left Quillen bifunctor with
respect to the (projective, injective) model structures, it
follows that $Q\Delta 1 * q \colon Q\Delta 1 * QS \to Q\Delta 1 * S$ is a weak equivalence between cofibrant objects.

Similarly, since all objects are cofibrant in $\SSet$, $p\colon Q \Delta 1
\to \Delta 1$ is a weak equivalence between injectively cofibrant
objects, whilst $QS$ is projectively cofibrant.  Therefore taking the
(injective, projective) choice, it follows that the left leg of the square $p * QS \colon Q\Delta 1 * QS \to \Delta 1 * QS$ is a weak equivalence of cofibrant objects.

Since weak equivalences are closed under $\lambda$-filtered colimits
in $\cm$, it is also true that $\Delta 1 * q\colon \Delta 1 * QS \to \Delta 1 * S$ is a weak equivalence.  Therefore, by 2-from-3 on the above square, the morphism $p * S \colon Q\Delta 1 * S \to \Delta 1 * S$ is a weak equivalence too.  Moreover since $\Delta 1 * S$ is fibrant, it is also cofibrant, so that $p*S$ is a weak equivalence between cofibrant objects.

Since $\cm$ is an enriched model category,  if $A \in \cm$ is fibrant
the map $$\cm(p * S,A) \colon \cm(\Delta 1 * S,A) \to \cm(Q\Delta 1 * S,A)$$ is a weak equivalence, whence so is the isomorphic $$[\cc\op,\SSet](\Delta 1,\ck(S-,A)) \to [\cc\op,\SSet](Q\Delta 1,\ck(S-,A))$$ as required.

\endproof

\begin{example}
The $\infty$-cosmos $\qCat$ of quasicategories arises in this way as
$\cm\fib$, where $\cm$ is $\SSet$ equipped with the Joyal model
structure. The $\infty$-cosmos CSS of complete Segal spaces arises in
this way as $\cm\fib$, where $\cm$ is a simplicial model structure on the
category of bisimplicial sets due to Rezk
\cite[Theorem~7.2]{Rezk-htyhty}.
The proposition can also be applied to the model category $\Cat$, with
the ``natural'' model structure: see
\cite[Proposition~1.2.11]{elements}; in this case, of course,
$\Cat\fib$ is just $\Cat$. For further examples, including
various models for $(\infty,n)$-categories, see \cite[Appendix~E]{elements}.
\end{example}

\begin{example}
  An example of a non-accessible $\infty$-cosmos is $\Cat\op$. As
  explained in \cite[Example~E.1.6]{elements}, this is an
  $\infty$-cosmos, with the injective-on-objects functors as
  isofibrations. But it is not accessible, since the underlying
  category is not accessible; indeed, if a category and its opposite
  are both accessible then the category must be a preorder: see \cite[Theorem~1.64]{AR}.
\end{example}

\section{First stability properties of accessible $\infty$-cosmoi}\label{sect:first}

 In Section 6 of \cite{elements}, Riehl and Verity show that a given $\infty$-cosmos gives rise to many others, such as the $\infty$-cosmos of isofibrations and slice $\infty$-cosmoi.  In the present section, we investigate a first group of these constructions, showing that they lift to the world of accessible $\infty$-cosmoi.

\subsection{The $\infty$-cosmos of isofibrations}

If $\ck$ is an $\infty$-cosmos, then $\ck^\isf$ becomes one too 
\cite[Proposition~6.1.1]{elements} on defining a commutative square
\begin{equation}\label{eq:fib-fib}  \xymatrix{
    A \ar@{>>}[d]_p \ar[r] & A' \ar@{>>}[d]^{p'} \\ B \ar[r] & B' }
\end{equation}
to be an an isofibration just when both the lower horizontal $B\to
B'$ and the induced map $A\to B\x_{B'}A'$
are isofibrations in $\ck$. 

\begin{proposition}\label{prop:fib-acc}
  If $\ck$ is an accessible $\infty$-cosmos, so is $\ck^\isf$, and the inclusion $\ck^{\isf} \to \ck^{\atwo}$ is an accessible functor. 
\end{proposition}

\proof
The fact that $\ck^\isf_0$ and the inclusion $\ck^{\isf}_0 \to \ck^{\atwo}_0$ are accessible is part of the definition
of $\ck$ being an accessible $\infty$-cosmos. Accessibility of the power functors holds because it holds in $\ck$, and powers and
sufficiently-filtered colimits in $\ck^\isf$ are computed pointwise.

Next, we need to show that $(\ck^\isf)^\isf_0$ is accessible and
accessibly embedded in $(\ck^\isf)^\two_0$. To see this, consider the pullback
\begin{equation*}
\xymatrix{
      (\ck^\isf)^\isf_0 \ar[r] \ar[d] & (\ck^\isf)^\two_0 \ar[d]^{(lh,pb)} \\
      \ck^\isf_0 \x \ck^\isf_0 \ar@{>>}[r] & \ck^\two_0 \x \ck^\two_0 }
 \end{equation*}
  in which the horizontal maps are the inclusions, and
  $(lh,pb)$  %the right vertical
  is the map sending a commutative square \eqref{eq:fib-fib}
  to the pair consisting of the lower horizontal $B\to B'$ and the  induced map $A\to B\x_{B'}A'$. 
  
The bottom leg of the pullback square is a product of two copies
of the accessible isofibration of categories $\ck^\isf_0\twoheadrightarrow
\ck^\two_0$ , and so is an 
accessible isofibration of categories.  Therefore, the claim will follow from Proposition~\ref{prop:pb-acc} if we can show that the right leg 
$$(lh,pb)\colon (\ck^\isf)^\two_0 \to (\ck)^\two_0 \times (\ck)^\two_0$$
 is accessible.  Using Lemma~\ref{lem:limits} we choose $\lambda$ such
 that, first, $\ck^\isf_0\to\ck^\two_0$ is closed under
 $\lambda$-filtered colimits; and second, pullbacks of isofibrations
 commute with $\lambda$-filtered colimits in $\ck$.  The first
 assumption ensures that $lh$ preserves $\lambda$-filtered colimits
 whilst the second assumption ensures that $pb$ does so too; hence so does $(lh,pb)$.
 
In addition to the above properties of $\lambda$, we now further
assume that $\lambda$-filtered colimits are homotopy colimits in
$\ck$.  We will show that the same property holds in $\ck^\isf$.  To
this end, let $\cc$ be $\lambda$-filtered and consider
$S\colon\cc \to \ck^\isf$, and let $Q \to \Delta 1 \in [\cc\op,\SSet]$ be a projective cofibrant replacement.  We must prove that 
\begin{equation}\label{eq:local1}
\xymatrix{
p^*\colon[\cc\op,\SSet](\Delta 1,\ck^{\isf}(S-,A)) \to [\cc\op,\SSet](Q,\ck^{\isf}(S-,A))
}
\end{equation} is an equivalence of quasicategories.  Now $S$ is
specified by its source and target components $S_0,S_1\colon\cc \to
\ck$ plus a natural pointwise isofibration $s\colon S_0 \to
S_1$, whilst $A$ is a single isofibration $a\colon A_0 \to A_1$, and by definition of $\ck^{\isf}$ we have a pullback square
\begin{equation*}
\xymatrix{
\ck^{\isf}(S-,A) \ar[r] \ar[d]  &\ck(S{_0}-,A_0) \ar[d]^{a_*} \\
\ck(S{_1}-,A_1) \ar[r]_{d^*} & \ck(S{_0}-,A_1) }
\end{equation*}
in $[\cc\op,\SSet]$ whose right leg is a pointwise isofibration.  Both the representables $[\cc\op,\SSet](\Delta 1,-)$ and $[\cc\op,\SSet](Q,-)$ preserve pullbacks, so that \eqref{eq:local1} is in fact the unique induced map between the pullbacks of the two horizontal rows below.

\begin{equation}\label{eq:spanmap}
\xymatrix{
(\Delta 1,\ck(S{_1}-,A_1)) \ar[d]^{p^{*}} \ar[r]^{(\Delta 1,d^*)} & (\Delta 1,\ck(S{_0}-,A_1))  \ar[d]^{p^{*}} & (\Delta 1,\ck(S{_0}-,A_0)) \ar[d]^{p^{*}} \ar@{>>}[l]_{(\Delta 1,a_*)} \\
(Q,\ck(S{_1}-,A_1)) \ar[r]_{(Q,d^*)} & (Q,\ck(S{_0}-,A_1))  & (Q,\ck(S{_0}-,A_0)) \ar@{>>}[l]^{(Q,a_*)} }
\end{equation}

Let us first observe that since $\lambda$-filtered colimits are homotopy colimits in $\ck$, each of the three vertical morphisms is an equivalence of quasicategories --- in particular, all of the objects in the diagram are fibrant in the Joyal model structure. Therefore, to prove that the induced map between the pullbacks is an equivalence of quasicategories, it will suffice by  Proposition C.1.13 of \cite{elements} to show that the two left-pointing morphisms are isofibrations.

Examining first the lower left-pointing morphism, observe that since $Q$ is projectively cofibrant and  $a_*\colon \ck(S{_0}-,A_1) \to \ck(S{_0}-,A_0)$ is a pointwise isofibration, it follows that $(Q,a_*)$ is an isofibration, as required.  The upper left-pointing morphism $(\Delta 1,a_*)$ is isomorphic to $\ck(\colim S_0,a):\ck(\colim S{_0},A_0) \to \ck(\colim S{_0},A_1)$, which is an isofibration since $a$ is one, completing the proof.
\endproof

\subsection{Slice constructions}

Slice categories can be a very convenient tool for expressing various
universal properties. This remains true in the
$\infty$-cosmos setting, but here the slice construction is based on
isofibrations with given codomain rather than arbitrary morphisms.

For a simplicially enriched category $\ck$ and an object $A\in\ck$, we
write $\ck\downarrow A$ for the enriched slice category: an object is
a morphism $p\colon B\to A$, while if $q\colon C\to A$ is also an
object then the corresponding hom is given by the pullback
\[ \xymatrix{
    (\ck\downarrow A)(p,q) \ar[r] \ar[d] & \ck(B,C) \ar[d]^{\ck(B,q)}
    \\
    1 \ar[r]_-{p} & \ck(B,A). } \]
In particular, a morphism in $\ck\downarrow A$ is just a commutative triangle.
If now $\ck$ is an $\infty$-cosmos, we write $\ck_{/A}$ for the full
subcategory of $\ck\downarrow A$ consisting of those $p\colon B\to A$
which are isofibrations. This $\ck_{/A}$ is also an $\infty$-cosmos
\cite[Proposition~1.2.22]{elements} on defining a morphism from
$p\colon B\twoheadrightarrow A$ to $q\colon C\twoheadrightarrow A$ to be an isofibration in
$\ck_{/A}$ just when the corresponding morphism $B\to C$ is one in $\ck$.

\begin{proposition}\label{prop:acc-slice}
  If $\ck$ is an accessible $\infty$-cosmos then so is $\ck_{/A}$ for
  each $A\in\ck$, and the inclusion $\ck_{/A}\to \ck\downarrow A$ is
  an accessible functor. 
\end{proposition}

\proof
In the pullback below left
 \[ \xymatrix{
    (\ck_{/A})_0 \ar[r]^{I_{0}} \ar[d] & \ck^\isf_0 \ar@{>>}[d]^{\cod_0} \\
    1 \ar[r]_-A & \ck_0 } \quad
    \xymatrix{
  (\ck_{/A})^\isf_0 \ar[r] \ar[d] &  (\ck^\isf)^\isf_0 \ar@{>>}[d]^{}
  \\ (\ck_{/A})^\atwo_0 \ar[r]_-{I^{\atwo}_0} & (\ck^\isf)^\atwo_0 } \]
the right vertical is an accessible isofibration between accessible
categories, and the lower horizontal an accessible functor between
accessible categories. It follows by Proposition~\ref{prop:pb-acc}
that $(\ck_{/A})_0$ is accessible and $I_0$ an accessible functor.
%, verifying Condition (1) in the definition of accessible $\infty$-cosmos.

Now consider the pullback above right.  Since $I_0$ is accessible, the Makkai-Par\'e Limit Theorem ensures that the lower horizontal $I_0^{\atwo}$ is also accessible.  The right vertical inclusion is an accessible functor between accessible categories by Proposition~\ref{prop:fib-acc} and moreover an isofibration.  Hence, by Proposition~\ref{prop:pb-acc} once again, the left vertical is an accessible functor between accessible categories, verifying Condition~\eqref{item:isofibs-accessible} in the definition of accessible $\infty$-cosmos.

For an object $p\colon B\to A$ of $\ck_{/A}$, the corresponding hom-functor
$\ck_{/A}((B,p),-)\colon\ck_{/A}\to\SSet$ can be constructed as a
pullback
\[ \xymatrix{
    \ck_{/A}((B,p),-) \ar[r] \ar[d] & \ck(B,\dom-) \ar[d] \\
    1 \ar[r]_-p & \ck(B,A) } \]
of functors $\ck_{/A}\to\SSet$, where the objects $1$ and $\ck(B,A)$
are seen as constant functors, while $\ck(B,\dom-)$ is the functor
sending $q\colon C\to A$ to $\ck(B,C)$, and the right vertical has
component at $q\colon C\to A$ given by
$\ck(B,q)\colon\ck(B,C)\to\ck(B,A)$. This is a pullback of accessible
functors, so is itself accessible, since pullbacks commute with
$\lambda$-filtered colimits in $\SSet$, for any infinite 
cardinal $\lambda$.  It follows by
Proposition~\ref{prop:enr-acc} that  Condition (1) in the definition of
accessible $\infty$-cosmos holds. 

The verification of Condition~\eqref{item:hty-colims} is identical in form to the corresponding verification in the proof of Proposition~\ref{prop:fib-acc},  the main difference is that the left vertical in \eqref{eq:spanmap} is replaced by the identity $1 \to 1$.

Finally, accessibility of the inclusion follows from the fact that
there is a pullback
\[ \xymatrix{
    (\ck_{/A})_0 \ar[r] \ar[d] & (\ck^\isf)_0 \ar@{>>}[d] \\ (\ck\downarrow A)_0 \ar[r]
    & (\ck^\two)_0 } \]
of accessible categories and accessible functors, in which the right
vertical is an isofibration.
\endproof

Later on, we will use the following simple result in our applications, and so record it now.

\begin{proposition}\label{prop:F/A}
  If $F\colon\cl\to\ck$ is an accessible cosmological functor, then
  so is the induced $F_{/A}\colon \cl_{/A}\to\ck_{/FA}$ for any $A\in\cl$.
\end{proposition}

\proof
There is a commutative square
\[ \xymatrix{
    \cl_{/A} \ar[r]^-{F_{/A}} \ar[d] & \ck_{/FA} \ar[d] \\
    \cl\downarrow A \ar[r]_-{F\downarrow A} & \ck\downarrow FA } \]
in which the vertical maps are the fully faithful inclusions. These
are accessible by Proposition~\ref{prop:acc-slice}, while $F\downarrow
A$ is so since it is the induced map between comma-categories (or comma objects)
in the $2$-category of accessible categories and accessible functors. Thus $F_{/A}$
is also accessible.
\endproof

\subsection{Dual $\infty$-cosmoi}
 
As described in Definition~1.2.25 of \cite{elements}, each $\infty$-cosmos $\ck$ has a dual $\infty$-cosmos $\ck\co$.  This has the same underlying category as $\ck$, with simplicial
homs given by $\ck\co(A,B)=\ck(A,B)\op$, and with the same
isofibrations as in $\ck$.  Powers in $\ck\co$ by $X$ are given by powers in
$\ck$ by the opposite simplicial set $X\op$.  

\begin{proposition}\label{prop:duality}
  If $\ck$ is an accessible $\infty$-cosmos, so is its dual $\ck\co$.
\end{proposition}

\proof
The only condition left to be verified is Condition~\eqref{item:hty-colims}, for which purpose we will investigate weighted colimits in $\ck$.  Let  $W\colon \cc\op \to \SSet$ a weight and $S\colon \cc \to \ck$ a diagram, which corresponds to a diagram $S\co \colon \cc\co \to \ck\co$.  For simplicity, let us suppose that $\cc$ is merely a category, so that $\cc\co = \cc$.  Applying the involution $(-)\op\colon \SSet \to \SSet$ levelwise gives an involution $(-)\op \colon [\cc\op,\SSet]_0 \to [\cc\op,\SSet]_0$.  We then have an isomorphism
$$\phi_{W,A}\colon [\cc\op,\SSet](W,\ck(S-,A))\op \cong
[\cc\op,\SSet](W\op,\ck\co(S\co-,A))$$ natural in $W$ and $A$.
Suppose now that $\cc$ is $\lambda$-filtered, that  $\lambda$-filtered colimits exist and are homotopy colimits in
$\ck$, and that $p\colon Q \to \Delta 1$ is a cofibrant replacement with $p$ a trivial fibration.  By the above, we have a commuting square
\begin{equation*}
\xymatrix @C2.5pc {
[\cc\op,\SSet](\Delta 1,\ck(S-,A))\op \ar[d]_{\phi_{\Delta 1,A}} \ar[r]^{{(p^{*})}\op} & [\cc\op,\SSet](Q ,\ck(S-,A))\op \ar[d]^{\phi_{Q,A}} \\
[\cc\op,\SSet](\Delta 1\op,\ck\co(S\co-,A)) \ar[r]^{({p\op)}^*} & [\cc\op,\SSet](Q\op,\ck\co(S\co-,A))
}
\end{equation*}
with vertical maps isomorphisms.  The upper horizontal is an
equivalence since the opposite of an equivalence is an equivalence, so
that the lower horizontal is one too.  Since this is induced by precomposition with $p\op\colon Q\op \to \Delta 1\op = \Delta 1$, we will have verified Condition~\eqref{item:hty-colims} if we can show that this is a cofibrant replacement of $\Delta 1$.  Indeed, this follows easily from the fact that the involution $(-)\op \colon [\cc\op,\SSet]_0 \to [\cc\op,\SSet]_0$ leaves the projective trivial fibrations unchanged, and so leaves the projectively cofibrant objects unchanged too. 
\endproof 

\subsection{Cosmological embeddings}\label{sect:embeddings}

Let $\ck$ be an $\infty$-cosmos. A (not necessarily full) simplicial
subcategory $\cl$ of $\ck$ is said to be {\em replete} \cite[Definition~6.3.1]{elements}
 if:
\begin{enumerate}
\item each object of $\ck$ equivalent to one in $\cl$ belongs to
  $\cl$;
\item each equivalence in $\ck$ between objects of $\cl$ belongs to
  $\cl$;
\item each $0$-arrow of $\ck$ isomorphic  to one in $\cl$ belongs to
  $\cl$;
\item the inclusion is full on positive-dimensional arrows.
\end{enumerate}
We also say that the inclusion $\cl\to\ck$ is replete.

\begin{remark}
  In the presence of  Condition (4), the other conditions amount to the
  fact that the 2-functor $h\cl\to h\ck$ is a fibration for the model
  structure of \cite{qm2cat} for 2-categories, there called an equiv-fibration.
\end{remark}

If moreover $\cl$ is closed in $\ck$ under cosmological limits, then it
can be made into an $\infty$-cosmos \cite[Proposition~6.3.3]{elements}
by defining a morphism in $\cl$ to
be an isofibration if and only if it is one in $\ck$. The inclusion
$\cl\to\ck$ is then said to be a {\em cosmological embedding}. 

\begin{lemma}\label{lem:fibration}
If $J\colon\cl \hookrightarrow \ck$ is a cosmological embedding, then each $J_{X,Y}\colon\cl(X,Y) \to \ck(X,Y)$ is an isofibration of quasicategories.
\end{lemma}
\begin{proof}
We must show that the $J_{X,Y}$ have the right lifting property with
respect to the inner horn inclusions and the inclusion $1 \to \iso$,
where $\iso$ denotes the (nerve of the) free-living isomorphism. 
The inner horn inclusions are bijective on vertices, while $J_{X,Y}$
is fully faithful on positive dimensional arrows, and these two
classes are orthogonal, so there are in fact unique liftings.
As for $1\to\iso$, we have a lifting problem
\begin{equation*}
\xymatrix{
1 \ar[d] \ar[r]^-{f} & \cl(X,Y) \ar[d]^{J_{X,Y}} \\
\iso \ar[r]^{} & \ck(X,Y)}
\end{equation*}
which amounts to giving a $0$-arrow $f \in \cl(X,Y)$ and a map $\iso
\to \ck(X,Y)$ which sends the isomorphism $0 \cong 1$ to an
isomorphism $f \to g$ in $\ck(X,Y)$; then by Condition (3) in the
definition of repleteness it follows that also $g\in\cl(X,Y)$. This shows that $\iso \to \ck(X,Y)$ factorizes through $\cl(X,Y)$ on $0$-simplices; since $J_{X.Y}$ is full on positive dimensional simplices it factorizes in all dimensions, as required.
\end{proof}

\begin{lemma}\label{lem:limitsarehomotopy}
Suppose that $\cl$ and $\ck$ are locally fibrant simplicially enriched
categories, and that $\cl \hookrightarrow \ck$ has each $\cl(X,Y) \to
\ck(FX,FY)$ an isofibration, injective in each dimension.  If
$\cl$ and $\ck$ have $W$-weighted colimits and they are homotopy
colimits in $\ck$, then they are also homotopy colimits in
$\cl$. There is also a dual result involving limits. 
\end{lemma}
\begin{proof}

Given $S\colon \cc \to \cl$, we can form $W * S \in \cl$.
For $p\colon Q \to W$ a cofibrant replacement, with $p$ a pointwise
trivial fibration, we will prove that $ p^*\colon[\cc\op,\SSet](\Delta 1,\cl(S-,A)) \to [\cc\op,\SSet](Q,\cl(S-,A))$ is a weak equivalence.  Consider the commutative square 

\begin{equation*}
\xymatrix{
[\cc\op,\SSet](W,\cl(S-,A)) \ar[d]_{{(J_{S-,A})}_*} \ar[rr]^{p^*} && \ar@{>>}[d]^{{(J_{S-,A})}_*} [\cc\op,\SSet](Q,\cl(S-,A))\\
[\cc\op,\SSet](W,\ck(JS-,JA)) \ar[rr]^{p^*} && [\cc\op,\SSet](Q,\ck(JS-,JA))}
\end{equation*}
in $\SSet$.  The map $J_{S-,A}\colon \cl(S-,A) \to \ck(JS-,JA)$ is a
 (pointwise) monomorphism and $p \colon Q \to W$ is a regular epimorphism
since it is a pointwise split epimorphism.  Thus by the (enriched) orthogonality of regular epimorphisms and monomorphisms, the above square is a pullback.

Since $J_{S-,A}\colon \cl(S-,A) \to \ck(JS-,JA)$ is a pointwise
isofibration between pointwise fibrant objects, and $Q$
is cofibrant in the projective model structure, it follows that the
right vertical arrow in the square is an isofibration of fibrant
objects.  Moreover, since the lower left object
$[\cc\op,\SSet](W,\ck(JS-,JA))$ is isomorphic to
$\ck(W * JS,JA)$, it is fibrant too, and the lower
horizontal is thus a weak equivalence of fibrant objects.  Now the
pullback of a weak equivalence between fibrant objects along a
fibration between fibrant objects is always a weak equivalence --- see Lemma A.2.4.3 of \cite{Lurie}, for example --- and so the upper horizontal is also a weak equivalence, as required.
\end{proof}

\begin{proposition}\label{prop:cosmo-accessible}
Suppose that $\cl \hookrightarrow \ck$ is a cosmological embedding
with $\ck$ an accessible $\infty$-cosmos.  If $\cl_0 \hookrightarrow
\ck_0$ is an accessible functor between accessible categories, then
$\cl$ is also an accessible $\infty$-cosmos in such a way  that the inclusion $\cl \hookrightarrow \ck$ is an accessible cosmological embedding.
\end{proposition}

\proof
The category $\cl_0$ is accessible by assumption.
Compatibility of powers and sufficiently filtered colimits holds in
$\cl_0$ since these are both calculated as in $\ck_0$.

By definition of the isofibrations in $\ck$, we have a pullback square
\[ \xymatrix{
    \cl^\isf_0 \ar[r] \ar[d] & \cl^{\atwo}_0 \ar[d] \\ %\Arr(\cl)_0 \ar[d] \\
    \ck^\isf_0 \ar@{>>}[r] & \ck^{\atwo}_0 } \]
of categories. 
The right vertical and lower horizontal are accessible
functors between accessible categories, and the lower horizontal is also an
isofibration, thus the upper horizontal is an accessible functor
between accessible categories by Proposition~\ref{prop:pb-acc}.

For Condition~\eqref{item:hty-colims}, let $\lambda$ be such that $\cl$ has $\lambda$-filtered colimits preserved by the inclusion to $\ck$ and such that $\lambda$-filtered colimits are homotopy colimits in $\ck$.  Then by Lemma~\ref{lem:limitsarehomotopy}, $\lambda$-filtered colimits are also homotopy colimits in $\cl$.
\endproof

\begin{proposition} \label{prop:cosmo}
  Suppose that $J\colon\cl\hookrightarrow\ck$ is an accessible
  cosmological embedding and $F\colon\ck'\to\ck$ is an accessible
  cosmological functor. Then in the pullback
  \[ \xymatrix{
      \cl' \ar[r]^-{G} \ar[d]_{J'} & \cl \ar[d]^{J} \\
      \ck' \ar[r]_-{F} & \ck } \]
  of simplicially enriched categories, $\cl'$ is an accessible
  $\infty$-cosmos, while
  $J'\colon\cl'\to\ck'$ is an accessible cosmological embedding, and
  $G\colon\cl'\to\cl$ is an accessible cosmological functor. 
\end{proposition}

\proof
By Proposition 6.3.12 of \cite{elements}, $G\colon\cl'\to\cl$ is a
cosmological embedding of $\infty$-cosmoi and $J'$ a cosmological functor.

A replete inclusion such as $J$ is in particular an isofibration at
the level of underlying categories. By Proposition~\ref{prop:pb-acc},
it follows that $(\cl')_0$ is an accessible category and $G_0$ and
$J'_0$ are accessible functors. Now $\cl'$ is an accessible
$\infty$-cosmos by Proposition~\ref{prop:cosmo-accessible}, and it
follows immediately that $J'$ is an accessible cosmological embedding
and $G$ is an accessible cosmological functor.
\endproof

\section{Left adjoint left inverses}\label{sect:lali}

A more exotic construction of $\infty$-cosmoi than those seen so far
is the $\infty$-cosmos of $\infty$-categories with limit of a given
shape --- see Section 6.3 of \cite{elements}.  As described therein,
such examples involving $\infty$-categorical structures with universal
properties, are naturally understood using the $\infty$-cosmos of
\emph{lalis}.  The present section adapts $\infty$-cosmoi of lalis to the accessible setting, and our results here make full use of all the axioms of an accessible $\infty$-cosmos. 

A morphism $f\colon A \to B$ in a $2$-category $\ck$ is said to be a left adjoint left inverse (lali) if it admits a right adjoint $u$ for which the counit $\epsilon \colon fu \Rightarrow 1_B$ is invertible --- the right adjoint $u$ is then called a right adjoint right inverse (rari).  In particular, a morphism $f$ is a lali just when it admits a rari, and vice versa.

A commutative square 

\begin{equation*}
\xymatrix{
A \ar[d]_{f} \ar[r]^{r} & A' \ar[d]^{f'} \\
B \ar[r]^{s} & B'
}
\end{equation*}
with $f$ and $f'$ lalis is said to be a morphism of lalis just when it also commutes with the right adjoints in the sense that its mate $ru \Rightarrow u's$ is invertible.  Note that this is independent of the choice of right adjoints.

Now if $\ck$ is an $\infty$-cosmos, a morphism $f\colon A \to B$ is said to be a lali/rari when it is one in the homotopy $2$-category $h\ck$.  Likewise, a commuting square $(r,s)\colon f \to f'$ is said to be a morphism of lalis if it is so in $h\ck$.  

In Proposition 6.3.10 of \cite{elements}, Riehl and Verity construct a
cosmologically embedded $\infty$-cosmos $$\lali(\ck) \hookrightarrow
\ck^{\isf}$$ whose objects are the isofibrations that are lalis 
(in other words, admit a rari)  and with morphisms the morphisms
of lalis.
The fact that it is a cosmological embedding fully determines the
remaining structure:  the inclusion reflects isofibrations and is full on positive-dimensional arrows.  

Let us mention an important point: by Lemma 3.6.9 of \cite{elements},
if an isofibration $f\colon A \twoheadrightarrow B$ is a lali, then
the right adjoint $u\colon B \to A$ can be chosen so that it is a
 section of $f$ and so that the counit is the identity $fu = 1$ in $h\ck$.

The goal of this section is to prove that if $\ck$ is an accessible $\infty$-cosmos then so is $\lali(\ck)$ with, moreover, the inclusion $\lali(\ck) \hookrightarrow \ck^{\isf}$ an accessible cosmological embedding.  This result, whose proof makes full use of all of the axioms for an accessible $\infty$-cosmos, is essential for our later applications.  

In moving towards this result, we begin by showing that lalis and their morphisms are \emph{representable} notions.

\begin{proposition}\label{prop:lali-char}
  Let $\ck$ be an $\infty$-cosmos.
  
 (1) An isofibration $p\colon A'\twoheadrightarrow A$ in $\ck$ is a lali if and only if
  \begin{enumerate}[(a)]
  \item $\ck(C,p)\colon\ck(C,A')\twoheadrightarrow\ck(C,A)$ is a lali in $\qCat$ for
    each $C\in\ck$;
  \item the square
    \[ \xymatrix{
        \ck(D,A') \ar@{>>}[d]_{\ck(D,p)} \ar[r]^-{\ck(c,A')} & \ck(C,A')
        \ar@{>>}[d]^{\ck(C,p)} \\
        \ck(D,A) \ar[r]_-{\ck(c,A)} & \ck(C,A) } \]
    defines a morphism of lalis in $\qCat$ for each $c\colon C\to D$
    in $\ck$. 
  \end{enumerate}
In fact the cases $C=A$ and $C=A'$ in (a), and $c=p$ in (b) suffice.

(2) Similarly, if $p\colon A'\twoheadrightarrow A$ and $q\colon B'\twoheadrightarrow B$ are lalis in $\ck$ then a
  morphism
  \[ \xymatrix{
      A' \ar[r]^-{f'} \ar@{>>}[d]_p & B' \ar@{>>}[d]^q \\ A \ar[r]_-f & B } \]
  in $\ck^\isf$ is a morphism of lalis in $\ck$ if and only if
  \[ \xymatrix{
      \ck(C,A') \ar[r]^-{\ck(C,f')} \ar@{>>}[d]_{\ck(C,p)} & \ck(C,B')
      \ar@{>>}[d]^{\ck(C,q)} \\
      \ck(C,A) \ar[r]_-{\ck(C,f)} & \ck(C,B) } \]
  is one in $\qCat$; and in fact the case $C=A$ suffices.
\end{proposition}

\proof
Any cosmological functor preserves lalis and morphisms of lalis, and
so in particular each representable $\ck(C,-)\colon\ck\to\qCat$ does
so. Similarly simplicially enriched natural transformations between cosmological functors
induce morphisms of lalis. Applying these facts to the cosmological
functors $\ck(C,-)\colon\ck\to\qCat$ and the natural transformations
$\ck(c,-)\colon\ck(D,-)\to\ck(C,-)$ gives the ``only if'' parts of the
proposition. We now turn to the converses. 

Suppose then that $\ck(C,p)$ is a lali in $\qCat$ if $C=A$ or $C=A'$,
and also that 
 \[ \xymatrix{
        \ck(A,A') \ar@{>>}[d]_{\ck(A,p)} \ar[r]^-{\ck(p,A')} & \ck(A',A')
        \ar@{>>}[d]^{\ck(A',p)} \\
        \ck(A,A) \ar[r]_-{\ck(p,A)} & \ck(A',A) } \]
 is a morphism of
lalis. Since $\ck(A,p)\colon  \ck(A,A')\twoheadrightarrow \ck(A,A)$ is a lali, the right
adjoint will send $1\colon A\to A$ to some $s\colon A\to A'$ with
$ps=1$, such that $p$ induces a bijection between  $2$-cells $x\to s$ and $px\to 1$ in $h\ck$  for any $x\colon A\to A'$.

Since the above square is a morphism of lalis, $p$ also induces a bijection
between  $2$-cells $y\to sp$ and $py\to p$,
for any $y\colon A'\to A'$.
In particular, the identity $p\to p$ corresponds to some $\sigma\colon
1\to sp$ with $p\sigma = 1$; on the other hand, the
images of $\sigma s, 1_s\colon s \rightrightarrows s$ under $p$ are  equal  to the
identity, so that  $\sigma s = 1_s$.  This
proves that $p$ is a lali, giving the ``if'' part of (1).

As for  (2), suppose that $\sigma\colon 1\to sp$ and $\sigma'\colon 1\to
s'q$ exhibit $p$ and $q$ as lalis, and that 
\[ \xymatrix{
    \ck(A,A') \ar[r]^-{\ck(A,f')} \ar@{>>}[d]_{\ck(A,p)} & \ck(A,B')
    \ar@{>>}[d]^{\ck(A,q)} \\
    \ck(A,A) \ar[r]_-{\ck(A,f)} & \ck(A,B) } \]
is a morphism of lalis in $\qCat$. We are to show that the induced
\[ \xymatrix{
    f's \ar[r]^{\sigma' f's} & s'qf's \ar@{=}[r] & s'fps \ar@{=}[r] &
    s'f } \]
is invertible, but this is just the component at $1_A$ of the induced
\[     h\ck(A,f')h\ck(A,s) \to h\ck(A,s')h\ck(A,f) \]
which is invertible by assumption.
\endproof

Consider a cosmological embedding $J\colon \cl \to \ck$ of
$\infty$-cosmoi and a diagram $S\colon \cc \to \cl$ such that
$\{W,JS\}$ exists in $\ck$.  Let us say that the cosmological
embedding {\em creates}\/ the weighted limit if $\{W,JS\} \in\cl$ and we have a pullback square \begin{equation*}
\xymatrix{
\cl(A,\{W,JS\}) \ar[d]_{J} \ar[r]^-{} &  [\cc,\SSet](W,\cl(A,S-)) \ar[d]_{J_{*}} \\
\ck(A,\{W,JS\}) \ar[r] & [\cc,\SSet](W,\ck(A,JS-)) 
}
\end{equation*}
natural in $A$.  This says precisely that the unit $W \to
\ck(\{W,JS\},JS-)$ factorizes through $J$ as $W \to
\cl(\{W,JS\},S-)$ and exhibits $\{W,JS\}$ as the weighted limit $\{W,S\}$ in $\cl$.

\begin{lemma}\label{lem:creation}
Each cosmological embedding $J\colon \cl \to \ck$ of $\infty$-cosmoi
creates any weighted limits that are homotopy limits in $\ck$.
\end{lemma}

\begin{proof}
Consider $W\colon \cc \to \SSet$ and $S\colon \cc \to \cl$ such that
$\{W,JS\}$ exists in $\ck$  and is a homotopy limit.  Let $p\colon Q \to W$ be a flexible
cofibrant replacement of $W$.  Since $\infty$-cosmoi have flexible
limits and cosmological functors preserve them, $\{ Q,S\}$
exists and is preserved by $J$; since $\{W,JS\}$ is a
homotopy limit, the canonical comparison $\lambda\colon \{W,JS\} \to
\{Q,JS\} =\{Q,S\}$ in $\ck$ is an equivalence.  Therefore, by repleteness, both $\{W,JS\} $ and $\lambda$ belong to $\cl$.  This allows us to consider the commutative diagram below.

\begin{equation*}
\xymatrix{
\cl(A,\{W,JS\}) \ar[d]_{J} \ar[r]^{\lambda_{*}} &  \cl(A,\{Q,JS\}) \ar[d]^{J} \ar[r]^-{\cong} & (Q,\cl(A,S-)) \ar[d]^{J_{*}} \\
\ck(A,\{W,JS\}) \ar[r]^{\lambda_{*}} & \ck(A,\{Q,JS\}) \ar[r]^-{\cong} & (Q,\ck(A,JS-)) \\
}
\end{equation*}

The right square is a pullback since its two horizontal components are
isomorphisms.  The verticals in the left square are full on
positive-dimensional arrows.  Therefore to show that the left square
is a pullback, it suffices to show that if $f\colon A \to
\{W,JS\}$ in $\ck$ has $\lambda \circ f \colon A \to \{Q,JS\}$ in
$\cl$, then $f$ is also in $\cl$.  Now by repleteness of
$\cl$, the equivalence-inverse $\lambda^{-1}$ is in $\cl$, which so is
the composite $\lambda^{-1}\circ\lambda\circ f$, and so finally the
isomorphic $f$. In particular the left square is a pullback, so that the outer square is a pullback.  Now its lower composite horizontal coincides with the corresponding morphism in the diagram below.

\begin{equation*}
\xymatrix{
\cl(A,\{W,JS\}) \ar[d]_{J} \ar[r]^-{\exists ! t_A} &  (W,\cl(A,S-)) \ar[d]^{J_*} \ar[r]^{p^*} & (Q,\cl(A,S-)) \ar[d]^{J_{*}} \\
\ck(A,\{W,JS\}) \ar[r]^{\cong} & (W,\ck(A,JS-)) \ar[r]^{p^*} & (Q,\ck(A,JS-)) \\
}
\end{equation*}
In this diagram, the right square is a pullback by the orthogonality
of the regular epimorphism $p\colon Q \to W$ and the monomorphism $J_{A,S-}\colon \cl(A,S-)
\to \ck(A,S-)$.  Therefore, by the universal property of the right
pullback square, we obtain a unique morphism $t_A$ to the pullback,
making the left square a pullback and such that the upper horizontals
of the two diagrams coincide.  Naturality and invertibility of the $t_A$ follows from the uniqueness of their construction.
\end{proof}

Using the lemma we next show that, in the accessible case, we can test for lalis and
their morphisms using small objects.

\begin{proposition}\label{prop:lali-representable}
  If $\ck$ is an accessible $\infty$-cosmos, for any sufficiently
  large $\lambda$, the canonical square
\[ \xymatrix{
    \lali(\ck)_0 \ar[r] \ar[d] & [\cg\op,\lali(\qCat)]_0 \ar[d] \\
    \ck^\isf_0 \ar[r] & [\cg\op,\qCat^\isf]_{0} }\]
is a pullback, where $\cg=\ck_{\lambda}$.
\end{proposition}
\begin{proof}
The lower horizontal sends $p\colon A_0\twoheadrightarrow A_1$ to $\ck(J-,p)\colon\cg\op\to \qCat^\isf$, where $J\colon\cg\to\ck$ is the
inclusion, and this lifts along the forgetful vertical functors to the upper
horizontal by virtue of Proposition~\ref{prop:lali-char}. 

To show that it is a pullback we need to show that, in the characterization of
Proposition~\ref{prop:lali-char}, it suffices to consider the case
where the objects $C$ and morphisms $c\colon C\to D$ lie in $\cg$.

Choose $ \lambda$ such that
\begin{itemize}
\item $\ck$ is $\lambda$-accessible as a simplicially enriched category;
\item $\lambda$-filtered colimits are homotopy colimits in $\ck$
\end{itemize}
and let $\cg=\ck_\lambda$.
 
Suppose then that $p\colon A_0\twoheadrightarrow A_1$ is an
isofibration in $\ck$ such that each $\ck(G,p)$ is a lali and each
square
\[ \xymatrix{
    \ck(H,A_0) \ar[r]^-{\ck(g,A_0)} \ar@{>>}[d]_{\ck(H,p)} &
    \ck(G,A_0) \ar@{>>}[d]^{\ck(G,p)} \\
    \ck(H,A_1) \ar[r]_-{\ck(g,A_1)} & \ck(G,A_1) } \]
is a morphism of lalis, for $g\colon G\to H$ in $\cg$.

For an arbitrary $C\in\ck$, we may write $C$ as a
$\lambda$-filtered colimit $\colim(S\colon \cj \to \ck)$ of a diagram taking values in $\cg$.  
Then, homming into $p$, we obtain $\ck(C,p) = \lim(\ck(S-,p)\colon
\cj\op \to \SSet^{\atwo}$ and since this diagram lifts to
$\ck(S-,p)\colon \cj\op \to \qCat^{\isf}$, and moreover since
$\ck(C,f)$ belongs to $\qCat^{\isf}$, it is also true that
$\ck(C,p) = \lim(\ck(S-,p)\colon \cj\op \to \qCat^{\isf}$).  We claim that this limit is in fact a homotopy limit.

To this end, let $p\colon Q \to \Delta 1 \in [\cj\op,\SSet]$ be a flexible cofibrant replacement.  By Proposition 6.2.8(i) of \cite{elements}, each $\infty$-cosmos admits flexible limits, so the limit $\{Q,\ck(S-,p)\}$ exists in $\qCat^{\isf}$.  Therefore, to show that $\ck(C,p)$ is the homotopy limit is equivalently to show that the induced map 
\[ \ck(C,p) \cong \{\Delta 1,\ck(S-,p)\} \to \{Q,\ck(S-,p)\} \]
is an equivalence in $\qCat^{\isf}$.  Since equivalences in $\qCat^{\isf}$ are pointwise as in $\qCat$, and since the above limits are pointwise --- the projections to $\qCat$ being cosmological --- this is equally to show that 
\[ \{\Delta 1,\ck(S-,A_i)\} \to \{Q,\ck(S-,A_i)\} \] is an equivalence of quasicategories for $i=0,1$, but this is simply
\[ [\cj\op,\SSet](\Delta 1,\ck(S-,A_i)\} \to  [\cj\op,\SSet](Q,\ck(S-,A_i)\} \]
which is an equivalence since $\lambda$-filtered colimits are homotopy colimits in $\ck$.

Thus $\ck(C,p)$ is a homotopy limit in $\qCat^\isf$ of lalis and morphisms of lalis. Therefore, by Lemma~\ref{lem:creation}, $\ck(C,p)$ is itself a lali and a limit in $\lali(\ck)$:  \emph{this means that the cone projections $\ck(C,p) \to \ck(S_j,p)$ are morphisms of lalis and jointly reflect morphisms of lalis}.

To complete the proof that $p$ is a lali, suppose now that $c\colon
C\to D$ is a morphism in $\ck$, and write $D$ as  a $\lambda$-filtered
colimit $\colim_j H_j$ of a diagram in $\cg$.
Each pre-composite $G_i\to C\to D$ of $c$ by a cocone inclusion $G_i \to C$ factorizes
through some $H_j\to D$, and now in the resulting diagram
\[ \xymatrix{
    \ck(D,p) \ar[r] \ar[d] & \ck(C,p) \ar[d] \\
    \ck(H_j,p) \ar[r] & \ck(G_i,p) } \]
 the vertical morphisms are morphisms of lalis since they are cone projections as above, whilst the lower horizontal is a morphism of lalis by assumption.  Hence the composite from top left to bottom right is a morphism of lalis.  Since the cone projections on the right vertical jointly reflect morphisms of lalis, it follows that the upper horizontal is a morphism of lalis too, as required.  Thus $p$ is a lali by Proposition~\ref{prop:lali-char}.
 
 Finally suppose that $p \to q$ is a morphism in $\ck^\isf$ where $p$ and $q$ are
lalis, and that the image of the square under $\ck(C,-)$ is a
morphism of lalis for each $C\in\cg$.  Write $C = \colim_j D_j$ as a $\lambda$-filtered colimit of objects of $\cg$.  Then for each cocone inclusion $D_j \to C$ we have the commutative square
\[ \xymatrix{
    \ck(C,p) \ar[r] \ar[d] & \ck(C,q) \ar[d] \\
    \ck(D_j,p) \ar[r] & \ck(D_j,q) } \]
 in which the vertical cone projections are morphisms of lalis and jointly detect morphisms of lalis.  By assumption the lower horizontal is a morphism of lalis, and it follows as before that the upper horizontal is a morphism of lalis too.  Therefore, by Proposition~\ref{prop:lali-char}, the original square is in fact a morphism of lalis in $\ck$, completing the proof.
\end{proof}

This last result essentially allows us to reduce to the case
$\ck=\qCat$, to which we now turn. For this, we need to work more
``analytically'', using a characterization of lalis in $\qCat$ from
\cite{elements}.

\begin{definition}
  Let $p\colon A'\twoheadrightarrow A$ be an isofibration in $\qCat$. Say that $a'\in
  A'$ is {\em $p$-universal}, or just universal if $p$ is understood, if for every diagram as in
  the solid part of \[ \xymatrix{
    1 \ar@/^1pc/[rr]^-{a'} \ar[r]_{[n]} & \partial\Delta[n] \ar[r]
    \ar[d] & A' \ar@{>>}[d]^p \\
    & \Delta[n] \ar[r] \ar@{.>}[ur] & A } \]
there exists a dotted arrow making the diagram commute.
\end{definition}

The following proposition illustrates the usefulness of this notion.

\begin{proposition}
  Consider a morphism
  \[ \xymatrix{
      A' \ar[r]^-{f'} \ar@{>>}[d]_p & B' \ar@{>>}[d]^q \\ A \ar[r]_-f & B } \]
  in $\qCat^\isf$.
  \begin{enumerate}
  \item $p$ is a lali if and only if for every $a\colon 1\to A$ there
    is a universal $a'\colon 1\to A'$ with $pa'=a$;
  \item if $p$ and $q$ are lalis, the square is a morphism of lalis if
    and only if $f'$ sends $p$-universal elements to $q$-universal elements.
  \end{enumerate}
\end{proposition}

\proof
(1) This is \cite[Lemma~F.3.1]{elements}.

(2) Suppose that $p$ and $q$ are lalis, and that $\sigma\colon 1\to
sp$ and $\tau\colon 1\to tq$ exhibit $s$ and $t$ as right adjoints to
$p$ and $q$.

Consulting the proof of \cite[Lemma~F.3.1]{elements}, one sees that
$s\colon A\to A'$ can be constructed in such a way that each $sa$ is
universal, and indeed any choice of universal lifts $sa$ of each $a\in
A$ can be assembled into an $s$. Thus it follows that $\sigma a'\colon
a'\to spa'$ is invertible if and only if  $a'$ is universal.

The square will be a morphism of lalis just when $\tau f's\colon f's\to
tqf's=tfps=tf$ is invertible; and this in turn will be the case if and
only if $\tau f'sa\colon f'sa\to tfa$ is invertible for each $a\in A$,
which by the previous paragraph amounts to the requirement that $f'$
 preserve universals.
\endproof

\begin{proposition}\label{prop:lali-qcat}
  $\lali(\qCat)\to\qCat^\isf$ is an accessible cosmological
  embedding. 
\end{proposition}

\proof It is a cosmological embedding by 
\cite[Proposition~6.3.10]{elements}.  Thus by
Proposition~\ref{prop:cosmo-accessible} it will suffice to show that
$\lali(\qCat)_0\to\qCat^\isf_0$ is accessible.

% Write $\Arr(\SSet_0)$ for the category of morphisms in $\SSet$ and
% commutative squares.
We know that $\qCat^\isf_0$ is accessible and
accessibly embedded in $\SSet^\two_0$, so it will suffice to show
that $\lali(\qCat)_0$ is an accessible category and the inclusion
$\lali(\qCat)_0\to\SSet^\two_0$ is an accessible functor. We do
so by showing that $\lali(\qCat)_0$ is a small injectivity class in
$\SSet^\two_0$, using techniques similar to those in Section~9 of
\cite{BourkeLackVokrinek}. 

Consider the category $1_{\Delta[0]}\vert\SSet^\two_0$ in which an object is a simplicial map $p\colon X'\to X$,
equipped with a subset $S\subseteq X'_0$ of ``marked objects'',
denoted $p\colon(X',S)\to X$; and a
morphism is a commutative square 
\[ \xymatrix{
    X' \ar[r]^-{f'} \ar[d]_p & Y' \ar[d]^q \\
    X \ar[r]_-f & Y 
  } \]
such that $f'$ sends marked objects to marked objects.

Then by  Proposition 9.2 of \cite{BourkeLackVokrinek},  the category $1_{\Delta[0]}\vert\SSet^\two_0$ is 
locally presentable and the forgetful functor to $\SSet^\two_0$ accessible, so, as per Corollary 9.3 of \cite{BourkeLackVokrinek},  the proof will be complete if we can show
that $\lali(\qCat)_0$ is a small injectivity class in $1_{\Delta[0]}\vert\SSet^\two_0$.
More precisely, we show that the collection of all those
$p\colon(X',S)\to X$ for which $p$ is both an isofibration and a lali, and $S$ consists precisely of all the
universal objects, is an injectivity class. Now injectivity with
respect to the diagrams 
\[ \xymatrix{ (\emptyset,\emptyset) \ar[r] \ar[d] &
    (\emptyset,\emptyset) \ar[d] \\ Y \ar[r]_-j & Z } \quad \quad \quad
  \xymatrix{ (Y,\emptyset) \ar[r]^-j \ar[d]_j & (Z,\emptyset) \ar[d]
    \\ Z \ar[r] & Z }
\]
for $j$  an inner horn inclusion or an endpoint inclusion $1 \to \iso$ says that $X$ is a
quasicategory, and $p$ an  isofibration (thus $X'$ is also a
quasicategory).

Injectivity with respect to the first of the following diagrams
\[ \xymatrix @C1.5pc { (\emptyset,\emptyset) \ar[r] \ar[d] & (1,1) \ar[d] \\ 1
    \ar[r] & 1 } ~
  \xymatrix @C1.5pc { (\partial\Delta[n],\{n\}) \ar[r] \ar[d] &
    (\Delta[n],\{n\}) \ar[d] \\ \Delta[n] \ar[r] & \Delta[n] } ~
  \xymatrix @C1.5pc { (\iso,\{0\}) \ar[r] \ar[d] & (\iso,\{0,1\}) \ar[d] \\
    (1,1) \ar[r] & (1,1) }
  \]
says that $S$ is non-empty; with respect to the second
(for all $n$) says that elements of $S$ are universal objects; and with respect to the third says that $S$ consists of all
the universal objects.
\endproof 

We can now put all the pieces together to prove the main result of the
section.

\begin{theorem}\label{thm:lali}
  If $\ck$ is an accessible $\infty$-cosmos then so is $\lali(\ck)$,
  and the cosmological embedding 
  $\lali(\ck)\to\ck^\isf$ is also accessible.
\end{theorem}

\proof
By Proposition~\ref{prop:cosmo-accessible}, it will suffice to show
that $\lali(\ck)_0$ is an accessible category and the inclusion
$\lali(\ck)_0\to\ck^\isf_0$ is an accessible functor.

By Proposition~\ref{prop:lali-representable} we have a pullback
\[ \xymatrix{
    \lali(\ck)_0 \ar[r] \ar[d] & [\cg\op,\lali(\qCat)]_0 \ar@{>>}[d] \\
    \ck^\isf_0 \ar[r] & [\cg\op,\qCat^\isf]_0 } \]
in which the right vertical is an isofibration of categories since the
inclusion 
$\lali(\qCat)_0\to\qCat^\isf_0$ is one. The lower horizontal and right
vertical are accessible functors between accessible categories, in the
case of the right vertical by Proposition~\ref{prop:lali-qcat}.
It now follows by Proposition~\ref{prop:pb-acc} that the left leg is
an accessible functor between accessible categories, as required.
\endproof

Dually, we have

\begin{corollary}
If $\ck$ is an accessible $\infty$-cosmos, then so is the
$\infty$-cosmos $\lari(\ck)$ of {\em right}\/ adjoint left inverses in
$\ck$, and the cosmological embedding $\lari(\ck)\to\ck^\isf$ is 
also accessible.
\end{corollary}

\begin{proof}
Reversing $2$-cells interchanges ralis and lalis --- in particular, $\lari(\ck) \to \ck^{\isf}$ is just $\lali(\ck\co)\co \to
((\ck\co)^{\isf})\co$.  The claim then follows from Theorem~\ref{thm:lali} combined with two applications of the $(-)^{\co}$ duality of Proposition~\ref{prop:duality}, on noting that since $(-)\co$ doesn't change underlying categories, it also respects accessibility of cosmological functors.
\end{proof}

We conclude this section with the observation that dealing with the
other two duals --- the laris and the raris --- would require an
alternative approach. While a 2-category $\ck$ possesses four duals, namely
$\ck\op$, $\ck\co$, $\ck\coop$, and $\ck$ itself, an $\infty$-cosmos
$\ck$ possesses only two: $\ck$ and $\ck\co$. Thus one cannot
simply define raris in $\ck$ to be lalis in $\ck\op$. In fact it seems
unlikely that one can even define an $\infty$-cosmos of laris or
raris in general. 

% \begin{example}
%   The pullback of the following diagram of laris and morphisms of
%   laris is not a lari.
%   \[ \xymatrix{
%       && \Delta[0] \ar[d]^{s_1} \ar[dl]_{s_0} \\
%       \Delta[0] \ar[r]^-{s_1} \ar[d]_{s_1} & \Delta[1] \ar[d]^{s_2}
%       & \Delta[1] \ar[dl]^{s_0} \\
%       \Delta[1] \ar[r]_-{s_2s_1d_1} & \Delta[2] } \]
% \end{example}

\section{Trivial fibrations and equivalences}\label{sect:TF}

Recall that the trivial fibrations in an $\infty$-cosmos are the
isofibrations which are also equivalences.  They are the objects of
a full subcategory $\TF(\ck)$ of $\ck^\isf$ which, following Proposition 6.1.5(ii) of \cite{elements},
is a cosmologically embedded $\infty$-cosmos.  To understand accessibility in this context, we can now follow the
same steps as we did when dealing with lalis. In fact the fullness
makes things easier, so we do not give all the details.

\begin{proposition}\label{prop:tf-representable}
  If $\ck$ is an accessible $\infty$-cosmos, for any sufficiently
  large $\lambda$, the canonical square 
\[ \xymatrix{
    \TF(\ck) \ar[r] \ar[d] & [\cg\op,\TF(\qCat)] \ar[d] \\
    \ck^\isf \ar[r] & [\cg\op,\qCat^\isf] } \]
is a pullback, where $\cg=\ck_\lambda$.
\end{proposition}

\proof
Choose $\lambda$ such that
\begin{itemize}
\item $\ck$ is $\lambda$-accessible as a simplicially enriched category;
\item $\lambda$-filtered colimits are homotopy colimits in $\cc$.
\end{itemize}
The lower horizontal is the fully faithful simplicial functor sending $p\colon A'\to
A$ to $\ck(J-,p)\colon\cg\op\to\SSet$, where $J\colon\cg\to\ck$ is the
inclusion. The vertical maps are fully faithful. The upper horizontal
exists (and is therefore fully faithful) because an isofibration $p$
in $\ck$ is a trivial fibration if and only if $\ck(C,p)$ is one for
all $C\in\ck$.
We need to prove that it will be one provided only that $\ck(C,p)$ is
one for $C\in\cg$. 

Suppose then that $\ck(G,p)$ is a trivial fibration for all $G\in\cg$,
and let $C\in\ck$ be arbitrary. We may write $C$ as a
$\lambda$-filtered colimit $\colim_iG_i$ of objects in $\cg$, and this
colimit is also a homotopy colimit. Thus $\ck(C,p)$ is a homotopy
limit of the trivial fibrations $\ck(G_i,p)$, and so is itself a
trivial fibration.
\endproof

Just as in the case of lalis, this last result now allows us to
restrict to the case of $\qCat$.

\begin{proposition}\label{prop:tf-qcat}
  $\TF(\qCat)\to\qCat^\isf$ is an accessible cosmological
  embedding. 
\end{proposition}

\proof It is a cosmological embedding by Proposition~\ref{prop:tf-representable}.
Thus by Proposition~\ref{prop:cosmo-accessible} it will suffice to show that
$\TF(\qCat)_0\to\qCat^\isf_0$ is accessible, or equivalently that
$\TF(\qCat)_0\to\qCat^\two_0$ is so. But this follows from the fact
that $\qCat_0$ is accessible, and that the trivial fibrations are the maps
with the right lifting property with respect to the boundary
inclusions $\partial\Delta[n]\to\Delta[n]$.
\endproof

\begin{theorem}\label{thm:TF}
  If $\ck$ is an accessible $\infty$-cosmos then so is $\TF(\ck)$, and
  the inclusion $\TF(\ck)\to\ck^\isf$ is an accessible cosmological embedding.
\end{theorem}

\proof
By Proposition~\ref{prop:cosmo-accessible} it will suffice to show
that $\TF(\ck)_0$ is an accessible category and the inclusion
$\TF(\ck)_0\to\ck^\isf_0$ is an accessible functor. This now follows
from Propositions~\ref{prop:pb-acc}, \ref{prop:tf-representable},
and~\ref{prop:tf-qcat}, just as in the proof of Theorem~\ref{thm:lali}.
\endproof 
 
An easy consequence of the above and the \emph{Brown factorisation lemma} is the following result.

\begin{proposition}\label{prop:axiomatics}
  Let $\ck$ be an accessible $\infty$-cosmos.  Then $\WE(\ck)_0$ is accessible and accessibly embedded in $\ck^\two_0$.
\end{proposition}

\proof

By
\cite[Proposition~1.2.19]{elements} there is a pullback 
\[ \xymatrix{
    \WE(\ck)_0 \ar[r] \ar[d] & \TF(\ck)_0 \ar@{>>}[d] \\
    \ck^\two_0 \ar[r]_R & \ck^\two_0 } \]
where the vertical maps are the (fully faithful) inclusions and $R$
is the functor which sends a morphism $f\colon A\to B$ to
$p_f\colon Pf\to B$, constructed via the pullback 
\[ \xymatrix{
Pf \ar[r] \ar[d]_{(q_f,p_f)} & B^\iso \ar@{>>}[d] \\
A\x B \ar[r]_-{f\x1} & B\x B. } \]
As usual, we now apply Proposition~\ref{prop:pb-acc}. The right
vertical is an isofibration, and is accessible by Proposition~\ref{prop:tf-qcat}, so we only
need to check that $R$ is accessible.  But this is constructed using
finite limits, and these commute with sufficiently filtered colimits
in any accessible category. 
\endproof

\begin{remark}\label{rmk:open-problem}
  This hints at a possible alternative definition of an accessible
  $\infty$-cosmos in which we replace Condition~\eqref{item:hty-colims} concerning homotopy colimits by an axiom asserting that $\WE(\ck)_0$ is accessible and accessibly embedded in $\ck^\two_0$.  The previous result ensures that our usual definition implies this second one.  If the second definition was equivalent to our usual one, it would be useful as many proofs concerning the stability of accessibile $\infty$-cosmoi would become shorter.  However, we have not been able to prove this, and we leave it as an open problem.
\end{remark}

\section{Applications to the motivating examples}\label{sect:Apps}

In this section we apply the results of the previous three sections to show that accessible $\infty$-cosmoi are stable under a whole host of key further constructions.

\subsection{$\infty$-categories with limits}

Let $J$ be a simplicial set. There is a cosmological functor
$F_J\colon \ck\to\ck^\isf$ sending $A\in\ck$ to the isofibration
$A^{J^\triangleleft}\to A^J$ given by restriction along the inclusion
$J\to J^\triangleleft=1\ast J$ of \cite[Notation~4.2.6]{elements}.

Then the pullback
\[ \xymatrix{
    \ck_{\top,J}\ar[r] \ar[d] & \lali(\ck) \ar[d] \\ \ck \ar[r]_-{F_J} &
    \ck^\isf.} \]
is, by \cite[Proposition~6.3.13]{elements} and its proof,  the $\infty$-cosmos $\ck_{\top,J}$ of $\infty$-categories in $\ck$ with
$J$-limits, and moreover $\ck_{\top,J} \to \ck$ is a cosmological embedding.

\begin{theorem}\label{thm:limits}
  If $\ck$ is an accessible $\infty$-cosmos, then so is the
  $\infty$-cosmos $\ck_{\top,J}$ of $\infty$-categories in $\ck$ with
  $J$-shaped limits. The cosmological embedding
  $\ck_{\top,J}\to\ck$ is then accessible as well. 
\end{theorem}

\proof
By  Proposition~\ref{prop:cosmo} and  Theorem~\ref{thm:lali}, it will suffice to show that the
cosmological functor $F_J$ is accessible, or equivalently that its
underlying ordinary functor $\ck_0\to (\ck^\isf)_0$ is accessible.
Now the fully faithful inclusion $(\ck^\isf)_0\to \ck^\two_0$ is
accessible, so it will suffice to show that the composite
$\ck_0\to\ck^\two_0$ preserves sufficiently filtered colimits. Since
these are formed pointwise in $\ck^\two_0$, we just need to know that
the two functors $\ck_0\to\ck_0$ sending $A$ to $A^{J\triangleleft}$
and to $A^J$ are accessible. This is true by Proposition~\ref{prop:enr-acc}.
\endproof

 Similar arguments apply to $\infty$-categories with a  set of limit shapes, whilst a dual argument applies to $\infty$-categories with colimits.  One can also combine limits and
colimits.

\subsection{Discrete objects}

An object $A$ of an $\infty$-cosmos $\ck$ is said to be {\em discrete}
\cite[Definition~1.2.26]{elements}  if the hom-quasicategory
$\ck(C,A)$ is in fact a Kan complex, for all $C\in\ck$. In particular,
the discrete objects of $\qCat$ are the Kan complexes.

By \cite[Proposition~6.1.6]{elements} and its proof, the discrete
objects of an $\infty$-cosmos $\ck$ form an $\infty$-cosmos
$\ck^\simeq$, whose inclusion into $\ck$ is a cosmological embedding.

\begin{theorem}\label{thm:discrete}
If $\ck$ is an accessible $\infty$-cosmos then so too is $\ck^\simeq$,
and the (fully faithful) inclusion is an accessible cosmological embedding. 
\end{theorem}

\proof
By \cite[Lemma~1.2.27]{elements}, an object $A$ is discrete if and
only if the isofibration $A^\iso\to A^\two$ is in fact a trivial
fibration. Thus we have a pullback square
\[ \xymatrix{
    \ck^\simeq \ar[r] \ar[d] & \TF(\ck) \ar[d] \\ \ck \ar[r]_-E &
    \ck^\isf } \]
as in \cite[Proposition~6.1.6]{elements}, in which the vertical maps
are fully faithful isofibrations, and $E$ sends $A$ to the projection $A^\iso\to
A^\two$.  By Proposition~\ref{prop:enr-acc}, $E$ is accessible; and by Theorem~\ref{thm:TF}, the right leg of the pullback square is an accessible cosmological embedding.  Therefore the result follows by Proposition~\ref{prop:cosmo}.
\endproof

\subsection{Cartesian fibrations}

By \cite[Proposition~6.3.14]{elements}, there exists a cosmologically embedded $\infty$-cosmos $\Cart(\ck) \hookrightarrow \ck^{\isf}$ consisting of those isofibrations which are \emph{cartesian fibrations} in $\ck$, and this 
can be obtained as the pullback below.
\[ \xymatrix{
    \Cart(\ck) \ar[r] \ar[d] & \lali(\ck) \ar[d] \\ \ck^\isf \ar[r]_-K
    & \ck^\isf } \]
As explained in the proof of \cite[Proposition~6.3.14]{elements}, this $K\colon\ck^\isf\to\ck^\isf$ is the cosmological functor sending an
isofibration $p\colon E\twoheadrightarrow B$ to the map $k$ occurring in the diagram 
\[ \xymatrix{
    E^\two \ar@/^1pc/[drr]^{p^\two} \ar@/_1pc/[ddr]_{\cod} \ar@{.>}[dr]^k \\
    & B/p \ar[r] \ar[d] & B^\two \ar@{>>}[d]^{\cod} \\
    & E \ar@{>>}[r]_-p & B } \]
in which the inner square is a pullback. 

%{ Should say something about the morphisms.}

\begin{theorem}\label{thm:Cart}
  If $\ck$ is an accessible $\infty$-cosmos then so is $\Cart(\ck)$,
  and the cosmological embedding $\Cart(\ck)\to\ck^\isf$ is also accessible.
\end{theorem}

\proof
By Theorem~\ref{thm:lali} once again it will suffice to show that
$K\colon \ck^\isf\to\ck^\isf$ is accessible. Just as in the proof of
Theorem~\ref{thm:limits}, this will be the case provided that the two
functors $\ck^\isf_0\to\ck_0$ sending $p\colon E\to B$ to $E^\two$ and
to $B/p$ are accessible.

The first of these is given by the domain functor $\dom\colon
\ck^\isf_0\to\ck_0$ followed by $\two\pitchfork-\colon\ck_0\to\ck_0$;
each of these is accessible, hence so is their composite.

Write $F\colon\ck^\isf_0\to\ck_0$ for the other functor, sending $p$
to $B/p$. It will be accessible if and only if
$\ck(C,F)\colon\ck^\isf_0\to\SSet$ is so, for each $C\in\ck$. And
$\ck(C,F)$ is the composite of the  inclusion
$\ck^\isf_0\to\ck^\two_0$, followed by the hom-functor
$\ck(C,-)\colon\ck^\two_0\to \SSet^\two$, followed by the functor
$K'\colon\SSet^\two\to\SSet$ sending $q\colon X\to Y$ in $\SSet$ to the
analogous $Y/q$. Each of these three functors is accessible:
\begin{itemize}
\item the inclusion by Condition~\eqref{item:isofibs-accessible} in the definition of accessible
  $\infty$-cosmos;
\item the hom-functor by Proposition~\ref{prop:enr-acc};
  \item $K'$ by Example~\ref{ex:adj-acc} and the fact that $K'$ is a
    right adjoint;
\end{itemize}
and so their composite is also accessible.
\endproof

Dually (see \cite[Proposition~6.3.14]{elements} once
again) we have:

\begin{theorem}\label{thm:coCart}
  If $\ck$ is an accessible $\infty$-cosmos then so is $\coCart(\ck)$,
  and the cosmological embedding $\coCart(\ck)\to\ck^\isf$ is also accessible.
\end{theorem}

 Similarly we can deal with the $\infty$-cosmoi $\DiscCart(\ck)$ of
discrete cartesian fibrations, 
and $\DisccoCart(\ck)$ of  discrete cocartesian fibrations. 

\begin{theorem}
  If $\ck$ is an accessible $\infty$-cosmos then so are
  $\DiscCart(\ck)$ and $\DisccoCart(\ck)$, and the inclusions into
  $\ck^\isf$ are accessible cosmological embeddings.
\end{theorem}

\proof
In the case of $\DiscCart(\ck)$, this follows from the existence of a
pullback
\[ \xymatrix{
    {}\DiscCart(\ck) \ar[r] \ar[d] & \TF(\ck) \ar[d] \\
    \ck^\isf \ar[r]^{K} & \ck^\isf } \]
as in the proof of \cite[Proposition~6.3.15]{elements}, where $K$ is
the map constructed at the beginning of this section.  (Alternatively, this can be deduced from Theorems~\ref{thm:Cart}
and~\ref{thm:discrete}, since $\DiscCart(\ck)$ is $\Cart(\ck)^\simeq$.)
The case of $\DisccoCart(\ck)$ is dual.
\endproof

\subsection{Fibrations with fixed base}

In this section we consider various flavours of fibration $A'\to A$
for a fixed object $A$ of our $\infty$-cosmos.

In the case of cartesian fibrations, for example, there are pullback
squares of simplicially enriched categories
\[ \xymatrix{
    \Cart(\ck)_{/A} \ar[r] \ar[d] & \Cart(\ck) \ar[d] \\
    \ck_{/A} \ar[r] \ar[d] & \ck^\isf \ar[d]^{cod} \\
    1 \ar[r]_-A & \ck } \]
as explained in the proof of \cite[Proposition~6.3.14]{elements}. As
also explained there, the horizontal maps are neither cosmological nor
replete, but the vertical maps are cosmological, and those in the
upper square are cosmological embeddings. Furthermore, at the level of underlying categories,
 the lower horizontal is accessible and the right vertical maps are accessible isofibrations, using
 Proposition~\ref{prop:fib-acc} and Theorem~\ref{thm:Cart} respectively.  Therefore both pullback squares consist of accessible categories
 and accessible functors, and by the universal property of the lower pullback square in the $2$-category of accessible categories, the cosmological
 embedding $\Cart(\ck)_{/A} \hookrightarrow \ck^{\isf}$ is accessible too.  Therefore, by Proposition~\ref{prop:cosmo-accessible}, $\Cart(\ck)_{/A}$ is an accessible $\infty$-cosmos and $\Cart(\ck)_{/A} \hookrightarrow \ck^{\isf}$ an accessible cosmological embedding.

This proves the first case of the following result, and the proofs for the other three flavours of fibration are similar.
\begin{theorem}\label{thm:one-sided}
  If $\ck$ is an accessible $\infty$-cosmos and $A$ an object of
  $\ck$, then each of the following is an accessible $\infty$-cosmos
  and the inclusion in $\ck_{/A}$ is an accessible cosmological
  embedding:
  \begin{itemize}
  \item $\Cart(\ck)_{/A}$
  \item $\coCart(\ck)_{/A}$
  \item $\DiscCart(\ck)_{/A}$
  \item $\DisccoCart(\ck)_{/A}$.
  \end{itemize}
\end{theorem}

\subsection{Two-sided fibrations}

If $\ck$ is an $\infty$-cosmos and $A,B\in\ck$, then there is an
$\infty$-cosmos $_{A\backslash}\Fib(\ck)_{/B}$ of 2-sided fibrations from $A$
to $B$, which is cosmologically embedded in $\ck_{/A\x
  B}$. Explicitly, this can be constructed as
\[ _{A\backslash}\Fib(\ck)_{/B} := \Cart(\coCart(\ck)_{/A})_{/A\x B\to A} \]
as explained in \cite[Section~7.2]{elements}. Similarly, by \cite[Section~7.4]{elements}, there is another cosmologically
embedded $\infty$-cosmos $_{A\backslash}\Mod(\ck)_{/B}$, given by the
discrete objects in $_{A\backslash}\Fib(\ck)_{/B}$.

\begin{theorem}
  Let $\ck$ be an accessible $\infty$-cosmos and $A,B\in\ck$. Then
  $_{A\backslash}\Fib(\ck)_{/B}$ and $_{A\backslash}\Mod(\ck)_{/B}$
  are also accessible $\infty$-cosmoi, and their inclusions into
  $\ck_{/A\x B}$ 
  % $_{\backslash  A}\Fib(\ck)_{/B}\to\ck_{/A\x B}$
  are accessible
cosmological embeddings. 
\end{theorem}

\proof  
 By Theorem~\ref{thm:one-sided} we know that $\coCart(\ck)_{/A}$ is an
accessible $\infty$-cosmos, and that $\coCart(\ck)_{/A}\to\ck_{/A}$ is
an accessible functor. By the same theorem, applied to the
$\infty$-cosmos $\coCart(\ck)_{/A}$ and the object $A\x B\to A$
therein, we know that the $\infty$-cosmos
\[ _{A\backslash}\Fib(\ck)_{/B}=\Cart(\coCart(\ck)_{/A})_{A\x B\to A}
\]
is accessible, and that it has an accessible cosomological embedding
into
% and its cosmological embedding into
$(\coCart(\ck)_{/A})_{/A\x B\to A}$. So it
will suffice to show that the functor
\[ (\coCart(\ck)_{/A})_{/A\x B\to A} \to \ck_{/A\x B} \]
is accessible.  This follows from Proposition~\ref{prop:F/A} applied
to the accessible cosmological functor $\coCart(\ck)_{/A}\to\ck_{/A}$
and the object $A\x B\to A$.

The case of $_{A\backslash}\Mod(\ck)_{/B}$ now follows by Theorem~\ref{thm:discrete}.
\endproof

% \bibliographystyle{plain}
% \bibliography{my}

\end{document}